\documentclass[11pt]{article}

\usepackage{amssymb, verbatim, amsmath}
\usepackage[all]{xy}
\usepackage{amsthm}
\usepackage[utf8]{inputenc}
\usepackage{dirtytalk}

\parskip \baselineskip

\newcommand{\bc}{\begin{center}}
\newcommand{\ec}{\end{center}}
\newcommand{\be}{\begin{enumerate}}
\newcommand{\ee}{\end{enumerate}}
\newcommand{\beq}{\begin{equation}}
\newcommand{\eeq}{\end{equation}}
\newcommand{\bi}{\begin{itemize}}
\newcommand{\ei}{\end{itemize}}
\newcommand{\bd}{\begin{description}}
\newcommand{\ed}{\end{description}}
\newcommand{\ba}{\begin{array}}
\newcommand{\bea}{\begin{eqnarray*}}
\newcommand{\eea}{\end{eqnarray*}}
\newcommand{\ea}{\end{array}}
\newcommand{\bt}{\begin{tabular}}
\newcommand{\et}{\end{tabular}}

\newcommand{\bmi}{\begin{minipage}}
\newcommand{\emi}{\end{minipage}}

\newtheorem{thm}{Theorem}[section]
\newtheorem{defn}[thm]{Definition}
\newtheorem{lem}[thm]{Lemma}
\newtheorem{pro}[thm]{Proposition}
\newtheorem{cor}[thm]{Corollary}
\newtheorem{algo}[thm]{Algorithm}

\newtheorem{exa}[thm]{Example}

\newtheorem{rem}[thm]{Remark}

\begin{document}

\bc {\bf\large On the generalized distributive set of a finite nearfield}\\[3mm]
{\sc Prudence Djagba  }

\it\small
Department of Mathematical Sciences,
Stellenbosch University,
South Africa\\
\rm e-mail: prudence@aims.ac.za
\ec
 
\normalsize

\quotation{\small {\bf Abstract:} For any nearfield $(R,+, \circ)$, denote by $D(R)$ the set of all distributive elements of $R$. Let $R$ be a finite Dickson nearfield  that arises from Dickson pair $(q,n)$.  For a given pair $(\alpha, \beta) \in R^2$ we  study the generalized distributive set  $ D(\alpha, \beta)= \big \{  \lambda \in R \thickspace \vert \thickspace (\alpha + \beta) \circ \lambda  = \alpha \circ \lambda + \beta \circ \lambda \big  \}$ where \say{$\circ $} is the multiplication of the Dickson nearfield. We find that $  D(\alpha, \beta)$ is not in general a  subfield of the finite field $\mathbb{F}_{q^n}$. In contrast to the situation for $D(R)$, we also find that $D(\alpha, \beta)$ is not in general a subnearfield of $R$. We obtain sufficient conditions on $\alpha, \beta$ for $ D(\alpha, \beta)$ to be a subfield of $\mathbb{F}_{q^n}$ and  derive an algorithm  that tests if $D(\alpha, \beta)$ is a subfield of $\mathbb{F}_{q^n}$ or not. We also  study the notions of $R$-dimension, $R$-basis, seed sets and seed number of $R$-subgroups of the Beidleman near-vector spaces $R^m$ where $m$ is a positive integer. Finally we determine the maximal $R$-dimension of $gen(v_1,v_2)$ for $v_1,v_2 \in R^m$,  where $gen(v_1,v_2)$ is the smallest $R$-subgroup containing the vectors $v_1$ and $v_2$.
  }

\small
{\it Keywords: Dickson nearfields, Beidleman near-vector spaces, $R$-subgroups, generalized distributive set.} \\
\normalsize

\textup{2010} \textit{MSC}: \textup{16Y30;12K05}
\section{Introduction and preliminaries}

 In $1905$  Dickson  wanted to know what structure arises if one axiom in the list of axioms for skewfields was weakened. He found that there  exist \say{nearfields}, which fulfill all
axioms for skewfields except one distributive law. Dickson achieved this by starting with a field and
changing the multiplication into a new operation \cite{dickson1905finite}.  Thirty years later  Zassenhauss classified all the finite nearfields in \cite{zassenhauss1935}. He found that all finite nearfields are either finite Dickson nearfields or one of the seven exceptional types. Since then Dancs \cite{susans1971,susans1972}, Karzel and Ellers \cite{ellerskarzel1964}, Zemmer \cite{zemmer1964} have solved some important problems in this area.

In $1996$ the first notion of near-vector spaces was introduced by Beidleman  \cite{beidleman1966near} in his PhD thesis. This notion generalises and extends the concept of a vector space to  obtain a certain structure which is no longer linear. He used nearring modules in the construction. Later Andr\'e \cite{andre1974lineare} introduced another  notion of near-vector spaces  and used automorphisms in the construction.  Andr\'e's version has been studied in many  papers, and theses, for example \cite{howell2007contributions} while Beidleman near-vector spaces had not since his thesis until recently  the authors of  \cite{djagbahowell18} added to the body of existing work. 

 In \cite{djagbahowell18}, the authors characterized the $R$-subgroups and subspaces of  finite dimensional Beidleman near-vector spaces. In this paper we continue  the work of that paper. We introduce the  notion of $R$-dimension, $R$-basis, seed set and seed number of an $R$-subgroup and find some properties. 

This paper is organized as follows. In subsection $1.1$ we have used the results by Hull and Dobell (\cite{hull1962random}) to explain the finite Dickson construction. In section $2$ we recall some results from \cite{djagbahowell18}  characterizing the $R$-subgroups of $R^m$. We attempt to give sufficient background for the paper  because of the nearfield experts that are new to finite dimensional Beidleman near-vector spaces. In section $3$ we  determine the distributive structure of finite Dickson nearfields. In section $4$ we give an application of the distributive elements to the concept of $R$-subgroups of $R^m$ where $R$ is a finite nearfield and  $m$ is a positive integer.  We evaluate the possible values of $R$-dimension for a given value of seed number and determine  the  seed number of $R^m$ where $m$ is a positive integer such that $ m \leq \vert R \vert+1$.

\subsection{Some remarks on finite nearfields}

Let $S$ be any group with identity  $0$. We will use $S^*$ to denote $S \setminus \{ 0\}.$

\begin{defn}(\cite{meldrum1985near}) Let  $(R,+,\cdot)$ be a triple such that $(R,+)$ is a group,
 $(R,\cdot)$ is a semigroup, and  $a \cdot (b+c)= a \cdot b+a \cdot c$ for all $a,b,c \in R.$ Then $(R,+,\cdot)$ is a (left) nearring. If in addition $ \big ( R^*, \cdot \big )$ is a group then $(R,+, \cdot)$ is called a nearfield. 
\end{defn}

So a nearfield is an algebraic structure similar to a skewfield (sometimes called a division ring) except that it has only one of the two distributive laws.  It is well known that the additive group of a (left) nearfield is abelian, see for instance \cite{pilz2011near}. Throughout this paper we will make use of (left) nearfields. To construct finite Dickson nearfields, we need two concepts: Dickson pair and coupling map.

\begin{defn} (\cite{pilz2011near})
A pair of positive integers  $(q,n)$ is called a Dickson pair if the following conditions are satisfied:
\begin{enumerate}
\item[(i)] $q$ is some power $p^l$ of some prime $p$,
\item[(ii)] each prime divisor of $n$ divides $q-1$,
\item[(iii)] $q \equiv 3$ $ \text{mod } 4$ implies $4$ does not divide $n$.
\end{enumerate}
\end{defn}

\begin{exa} The following are Dickson pairs: $(7,9),(3,2),(4,3),(5,4)$ and $(5,8).$
\end{exa}
Let $(q,n)$ is a Dickson pair  and $k \in \{1, \ldots, n \}.$ We will denote the positive integer  $\frac{q^k-1}{q-1}$ by $[k]_q$. 

 \begin{defn}(\cite{pilz2011near})
Let $R$ be a nearfield and $\textit{Aut} (R,+,\cdot ) $ the set of all automorphisms of $R$. A map $\phi: \thickspace R^* \to  \textit{Aut} (R,+,\cdot ) $  defined by $ n \mapsto \phi_n$
is called a coupling map if for all $n,m \in R^*, \thickspace \phi _n \circ \phi_m= \phi _{ \phi _n (m) \cdot n}.$
\end{defn}

Furthermore, to generate a sequence of numbers which at least appear to be drawn at random from a certain probability distribution (uniform or normal, poisson, or some other), we begin with a positive integer $m$, called the modulus,  a positive integer $x_0,$ called starting value such that $0 \leq x_0 <m$, an integer $a,$ called the multiplier such that $0 <a <m$ and another integer $c$, called the increment such that $0 \leq c < m$. We then define a sequence $\{x_i \}$ of positive integers, each less than $m,$ by means of the congruence relation
\begin{align}
x_i\equiv (a x_{i-1} +c) \mod m.
\label{con}
\end{align}
 For the case $c \neq 0$, an important number theoretic property has been discovered  by Hull and Dobell in $1962$ and is stated as the following:
\begin{thm}(\cite{hull1962random})
The sequence defined by the congruence relation (\ref{con}) has full period $m,$ provided that
\begin{itemize}
\item[(i)] $c$ is relatively prime to $m,$
\item[(ii)] $a \equiv 1  \mod p $ if $p$ is a prime  factor of $m,$
\item[(iii)]  $a \equiv 1  \mod 4 $ if $4$ is a factor of $m.$ 
\end{itemize}
\label{hull}
\end{thm}
We now deduce the following well known result as a special case  of Hull and Dobell's theorem.
\begin{lem}(\cite{pilz2011near,wahling1987theorie})
Let $(q,n)$ be a Dickson pair. Then $ \big \lbrace  [k]_q: 1 \leq k \leq n  \big  \rbrace $ forms a finite complete set of different residues  modulo $n.$
\label{lehull}
\end{lem}
\begin{proof}
We take $x_0=1, \thickspace a=q, \thickspace m=n$ and $c=1$
into the relation (\ref{con})  and apply  Theorem \ref{hull}. Thus we have the following: $x_0=1, x_1=q+1, x_2=q^2+q+1, \ldots, x_{n-1}= q^{n-1}+ q^{n-2}+ \ldots+1$. Hence the sequence $\{ x_i \}_{i=0, \ldots, n-1}$ has full period $n$ and are all different  residues modulo $n.$ Thus $\{ x_0, \ldots, x_{n-1} \}$  is $ \big \lbrace  [k]_q: 1 \leq k \leq n  \big  \rbrace $. 
 \end{proof}
Every Dickson pair $(q,n)$ gives rise to a finite Dickson nearfield. This is obtained by replacing the usual multiplication \say{$ \cdot$} in the finite field $\mathbb{F}_{q^n}$  of order $q^n$ with a new multiplication \say{$\circ$}. We shall denote the set of Dickson nearfields  arising from the Dickson pair $(q,n)$  by $DN(q,n)$ and  the  Dickson nearfield arising from the Dickson pair $(q,n)$ with generator $g$ by $DN_g(q,n)$. Furthermore in \cite{pilz2011near}   the new multiplication is constructed as follows: 

Let $g$ be such that $ \mathbb{F}_{q^n}^*= \langle g \rangle$ and $H = \langle g^n \rangle $. The quotient group is given by
 \begin{align*}
  \mathbb{F}_{q^n}^* / H & = \big \lbrace  g^{[1]_q}H, g^{[2]_q}H,\ldots, g^{[n]_q}H \big \rbrace \\
  & =\big \lbrace  H, gH,\ldots, g^{n-1}H \big \rbrace \thickspace \mbox{by Lemma \ref{lehull}} .
 \end{align*}
 The  coupling map $\phi$ is defined as
 \begin{align*}
\begin{array}{lcl}
\mathbb{F}_{q^n}^*& \to & \textit{Aut}(\mathbb{F}_{q^n},+,\cdot) \\
\alpha  & \mapsto & \phi_{\alpha}= \varphi^k(\alpha)  
 \end{array} 
  \end{align*} where  $\varphi$ is the Frobenius automorphism of $\mathbb{F}_{q^n}$ and $ k$ is a positive integer $( k \in \{1,\ldots,n \})$  such that $ \alpha \in g^{[k]_q}H  $. Let $\alpha, \beta  \in \mathbb{F}_{q^n},$ the  we have 
  \begin{align*}
\alpha \circ  \beta &= \left\{
\begin{array}{lcl}
\alpha \cdot \phi_{\alpha}(\beta) & \text{if} &  \alpha\ \neq 0  \\ 
0 & \text{if} & \alpha=0
\end{array}\right.  \\
&=
\begin{cases}
\alpha \cdot  \varphi^k (\beta) \thickspace  \thickspace \text{if $ \alpha \in g^{[k]_q}H $} \\
 0     \quad \quad \quad \quad  \text{if}  \thickspace     \alpha=0 
\end{cases} \\
  &= 
\begin{cases}
\alpha	 \cdot \beta^{q^k}   \thickspace \thickspace  \text{if $ \alpha \in g^{[k]_q} H  $} \\
0  \quad \quad \quad  \text{if}  \thickspace    \alpha=0  
\end{cases} 
\end{align*}
for $k \in \{ 1,\ldots,n \}$. Thus $DN_g(q,n):=\big ( \mathbb{F}_{q^n}, +, \circ \big )$ is the finite Dickson nearfield constructed by taking $H = \langle g^n \rangle$. By taking all Dickson pairs, all finite Dickson nearfields arise in this way.  Furthermore we deduce the following.
\begin{lem}
Let $(q,n)$ be a Dickson pair with $q=p^l$ for some prime $p$ and positive integers $l,n$. Let $g$ be a generator of $\mathbb{F}_{q^n}^*$ and  $R$  the finite nearfield constructed with $H = \big < g^n \big >.$ If  $n$ divides $q-1$ then $g^{[i]_q}H=g^iH$ for all $i=1,\ldots,n$.
\label{lenn}
\end{lem}
\begin{proof}
Suppose that $n$ divides $q-1$. We need to show that $[i]_q \equiv i \mod n$ and we  proceed by induction. For $i=1.$ We have $[1]_q \equiv 1 \mod n$. Suppose that $[i]_q \equiv i \mod n$. Since  $n$ divides $q-1$ then  $q^i \equiv 1 \mod n.$ So $[i+1]_q \equiv i+1 \mod n.$
\end{proof}
Note that the conclusion of Lemma \ref{lenn} is not always true for any Dickson pair, for example, take $(q,n)=(7,9)$. We have $[2]_7=8$ but $8 $ is not congruent to $2$ modulo $9.$ The use of Theorem \ref{hull} taken from  \cite{hull1962random} allows us to give an alternate proof to that found in \cite{wahling1987theorie} of the following lemma.
 
 \begin{lem}(\cite{wahling1987theorie})
 Let $(q,n)$ be a Dickson pair with $q=p^l$ for some prime $p$ and positive integers $l,n$. Let $g$ be a generator of $\mathbb{F}_{q^n}^*$ and  $R$  the finite nearfield constructed with $H = \big < g^n \big >.$ Then $n$ divides $[n]_q$ and $g^{[n]_q}H=H$.
 \end{lem}
\begin{proof}
Suppose that $(q,n)$ is a  Dickson pair.
Let us consider $x_0=1, a=q, m=n$ and $c=1.$ By Theorem \ref{hull}  the period of the sequence $\{x_i \}$ is exactly $n$. Thus $1=x_0$ and is equivalent to $x_n$. But we have $x_{n-1}=[n]_q$ which satisfies the recurrence 
\begin{align*}
 x_n \equiv q x_{n-1}+1 \mod n  \Leftrightarrow  1 \equiv q x_{n-1}+1 \mod n  \Leftrightarrow q x_{n-1} \equiv 0 \mod n.
\end{align*}
Since every prime divisor of $n$ divides $q$ then $\textsc{gcd} (q,n)=1$. Thus 
\begin{align*}
q x_{n-1} \equiv 0 \mod n \Leftrightarrow x_{n-1} \equiv 0 \mod n.
\end{align*}
Since  $n$ divides $[n]_q$ then $g^{[n]_q}H=H$.
\end{proof}
Thirty years after Dickson’s work,  Zassenhauss fundamentally determined all finite nearfields.
\begin{thm}\cite{zassenhauss1935}
A finite nearfield is either a finite Dickson nearfield or it is one of
the $7$ exceptional nearfields of order $5^2 , 7^2 , 11^2 , 23^2 , 29^2, 59^ 2$ .
\end{thm}
Note that there exist two exceptional nearfields of order $11^2$ (see \cite{zassenhauss1935} for more details).
\subsection{Beidleman near-vector spaces}

 The concept of a ring module can also be  extended to a more general concept called  a nearring module where the set of scalars is taken to be a nearring.
\begin{defn}
An additive group $(M,+)$ is called a (right) nearring module over a (left) nearring $R$ if there exists a mapping,
\begin{align*}
\eta: \thickspace & M \times R \to M \\
& (m,r) \to mr
\end{align*} such that $m(r_1+r_2)=mr_1+mr_2$ and $m(r_1r_2)= (mr_1)r_2$ for all $r_1,r_2 \in R$ and $m \in M.$

We write $M_R$ to denote that $M$ is a  (right) nearring module over a (left) nearring  $R$.
\end{defn}
\begin{defn}(\cite{beidleman1966near})
A subset $H$ of a nearring module $M_R$ is called an $R$-subgroup if  $H$ is a subgroup of $(M,+)$ and  $HR= \lbrace hr \vert h \in H, r \in R \rbrace \subseteq H. $
\end{defn}
\begin{defn} (\cite{beidleman1966near})
A nearring module $M_R$ is said to be irreducible if $M_R$ contains no proper $R$-subgroups. In other words, the only $R$-subgroups of $M_R$ are $M_R$ and $\lbrace 0 \rbrace.$
\end{defn}
\begin{defn} (\cite{beidleman1966near}) Let $M_R$ be a nearring module.
$ N  $ is a submodule  of $M_R$ if 
\begin{enumerate}
\item[(i)] $ (N,+)$ is a normal subgroup of $(M,+),$
\item[(ii)] $(m+n)r-mr \in N$ for all $m \in M, n \in N$ and $r \in R.$
\end{enumerate}
\end{defn}

\begin{pro}(\cite{beidleman1966near})
Let $N$ be a submodule of $M_R.$ Then $N$ is an $R$-subgroup of $M_R.$
\label{pro}
\end{pro}
 Note that the converse of this proposition is not true in general. In his   thesis Beidleman gives a counter example. 
\begin{thm}(\cite{beidleman1966near})
Let $R$ be a nearring that contains the  identity element $1 \neq 0.$ $R$ is a nearfield if and only if $R$ contains no proper $R$-subgroups.
\label{irre}
\end{thm}
\begin{rem}
Let $R$ be a nearfield. By Theorem \ref{irre}, $R_R$  is an irreducible $R$-module. Thus $R$ contains no proper $R$-subgroups, which means that $R$ contains only $\{ 0 \}$ and $R$ as submodules of $R_R.$
\end{rem}
In the following, let $\lbrace M_i  \rbrace_{ i \in I} $  be a collection of submodules  of the nearring module $M_R$.
\begin{defn}
Suppose that $M_R= \sum _{i \in I} M_i.$ Let $m \in M_R.$ Then $m$ is called sum of elements of the submodules $M_i$ if there exists $m_i \in M_i$ for all $i \in I$ such that $m= \sum _{i \in I}m_i.$
\end{defn}
\begin{defn}(\cite{beidleman1966near})
  $M_R$ is said to be a direct sum of the submodules $\lbrace M_i  \rbrace_{ i \in I} $ if  the additive group $(M,+)$ is a direct sum of the normal subgroups $ (M_i,+),$   $ i \in I $. In this case we write $M_R= \bigoplus _{i \in I} M_i.$
\end{defn}

\begin{pro}(\cite{beidleman1966near}) Suppose that
$M_R= \sum _ {i \in I} M_i$. Then every element of $M_R$ has a unique representation as a  sum of elements of  the submodules $M_i$ if and only if $ M_k \cap \sum _{i \in I, i \ne k} M_i= \lbrace 0 \rbrace$ for all $k.$
\end{pro}
\begin{pro}(\cite{beidleman1966near})
 $M_R = \bigoplus _{i \in I} M_i$ implies that $M_R= \sum_{i \in I } M_i$ and  $m_i+ m_j=m_j+m_i$ for all $m_i \in M_i, m_j \in M_j$ such that $i \neq j$. 
\label{rp}
\end{pro}

\begin{lem}(\cite{beidleman1966near}) Let $M_R = \bigoplus _{i \in I} M_i$ where   $M_i$ is a submodule of $M_R.$ If $m =\sum_{i \in I} m_i$ where $m_i \in M_i$ and $r \in R$ then
\begin{align*}
mr= \big ( \sum_{i \in I} m_i \big ) r= \sum_{i \in I}( m_ir).
\end{align*}
\label{lemm}
\end{lem}

We are now ready to define a Beidleman near-vector space.
\begin{defn}(\cite{beidleman1966near}) Let $(M_R,+)$ be a group and $R$  a nearfield.
$M_R$ is called a Beidleman near-vector space if $M_R$ is a nearring module which is a direct sum of irreducible submodules.
\end{defn}
\section{Description of the $R$-subgroups of $R^m$}
The material presented in this section is  taken from \cite{djagbahowell18}, in which    finite dimensional Beidleman near-vector spaces were characterized as follows:
\begin{thm}(\cite{djagbahowell18})
Let $R$ be a (left) nearfield and $M_R$  a (right) nearring module. $M_R$ is a finite dimensional near-vector space if and only if $M_R \cong R^m$ for some   positive integer $m = \dim (M_R).$
\label{thm1}
\end{thm}
From now on we will choose $R$ to be  a finite nearfield.
\begin{defn}(\cite{djagbahowell18})
Let $v_1,v_2,\ldots,v_k $ be a finite number of vectors in $R^m$. The smallest $R$-subgroup of $R^m$ containing $ v_1,v_2,\ldots,v_k $ is denoted by $gen(v_1,\ldots,v_k).$
\end{defn}

Let $LC_0(v_1,v_2,\ldots,v_k):=\{ v_1,v_2,...,v_k\}$ and for $n\geq0$, let $LC_{n+1}$ be the set of all linear combinations of elements in $LC_n(v_1,v_2,\ldots,v_k)$, i.e.
\begin{equation*}
	LC_{n+1}(v_1,v_2,\ldots,v_k)=\left \{ \sum_{w \in LC_n} w \lambda_w  \thickspace | \thickspace \lambda_w \in R  \thickspace \forall w \in LC_n \right\}.
\end{equation*}
We shall  denote $LC_n(v_1,v_2,\ldots,v_k)$ by $LC_n$ for short when there is no ambiguity with regard to the initial set of vectors. The following theorem gives an explicit  description of $gen(v_1,\ldots,v_k)$.
\begin{thm}(\cite{djagbahowell18}) Let $v_1,v_2,\ldots,v_k \in R^n$.
	We have 
	\begin{equation*}
	gen(v_1,\ldots,v_k)=\bigcup_{i=0}^\infty LC_i.
	\end{equation*}
	\label{th1}
\end{thm}

In order to give a  description of $gen(v_1,\ldots,v_k)$ in terms of the basis elements, the authors of \cite{djagbahowell18} first derived results analogous to those  for  row-reduction in vector spaces.
\begin{lem}(\cite{djagbahowell18})
\begin{itemize}
\item For any permutation $\sigma$ of the indices $1,2,...,k$, we have
	\begin{equation*}
	gen(v_1,\ldots,v_k)=gen(v_{\sigma(1)} ,\ldots,v_{\sigma(k)}).
	\end{equation*}
\item 	If $ 0 \neq \lambda \in R $, then
	\begin{equation*}
	gen(v_1,\ldots,v_k)=gen(v_1  \lambda,\ldots,v_k).
	\end{equation*}
\item 	For any scalars $\lambda_2, \lambda_3,\ldots , \lambda_k \in R$, we have 
	\begin{equation*}
	gen(v_1,\ldots,v_k)=gen \big (v_1 + \sum_{i=2}^{k} v_i \lambda_i,v_2, \ldots,v_k \big ).
	\end{equation*}
\item If $w \in gen(v_1,\ldots,v_k)$ then
	\begin{equation*}
	gen(v_1,\ldots,v_k)=gen(w, v_1,v_2, \ldots,v_k).
	\end{equation*}	
\end{itemize}
\label{l1}
\end{lem}

\begin{defn}
Let $v \in R^m$. Then $v$ is called an $R$-linear combination of some finite set of vectors $v_1,\ldots,v_k \in R^n$ if $v \in gen(v_1,\ldots,v_k).$
\end{defn}
\begin{defn}
The $R$-row space of a matrix is the set of all  $R$-linear combinations of its row vectors.
\end{defn}
\newpage
\begin{rem}~
\begin{itemize}
\item Let $w \in R^m.$ Note that $w$ is a linear combination of some finite set of vectors $v_1,\ldots,v_k \in R^m$ if $w \in LC_1(v_1,\ldots,v_k).$ If any vector $w$ is a linear combination of the finite  set of vectors $v_1,\ldots,v_k \in R^m$ then $w$ is also  an $R$-linear combination of the finite set of vectors $v_1,\ldots,v_k \in R^m$. But the converse is not true.
\item 
A matrix $V$ consisting of the rows  $v_1,\ldots,v_k \in R^m$  will be denoted by $V=\big ( v_i^j \big )_ {\substack{ 1 \leq i \leq k \\ 1 \leq j \leq m}}$ where $v_i^j$ is always the $j$-th entry of $v_i$.
\end{itemize} 

\end{rem}

The following result gives the characterization of the $R$-subgroups of $R^m,$ where $R$ is a finite nearfield. We include the proof (taken from \cite{djagbahowell18}) since we will make use of the intermediate steps in the proof throughout the next section.
\begin{thm}(\cite{djagbahowell18}) Let  $v_1,\ldots,v_k$ be vectors in $R^m$. Then
\begin{align*}
gen(v_1,\ldots,v_k)= \bigoplus_{i=1}^{k'}u_iR,
\end{align*}where the $u_i$ (obtained from $v_i$ by an explicit
procedure) for $i \in \lbrace 1, \ldots,k' \rbrace $ are the rows of  a matrix  $U=\big(u_{i}^j  \big) \in R^{k' \times m}$ each of whose columns has at most one non-zero entry.
\label{th2}
\end{thm}
\begin{proof}
Given a set of vectors $v_1,\ldots,v_k \in R^m$, arrange them in a matrix $V$ whose $i$-th row is $v_i$. Say $V=\big ( v_i^j \big )_ {\substack{ 1 \leq i \leq k \\ 1 \leq j \leq m}}$. Then $gen(v_1,\ldots,v_k)$ is the $R$-row space of $V$, which is an $R$-subgroup of $R^m$. We can then do the usual Gaussian  elimination on the rows. According to  Lemma \ref{l1}, the $gen$ spanned by the rows will remain unchanged under each row operation (swapping rows, scaling rows, adding multiples of a row to another). When the algorithm terminates, we obtain a matrix $W \in R^{k \times m }$ in reduced row-echelon form (denoted by $RREF(V)$). Let the non-zero rows of $W$ be denoted by $w_1,w_2,\ldots,w_t$ where $t \leq k$.

Case 1: Suppose that every column has at most one non-zero entry. Then
\begin{equation*}
	gen(v_1,\ldots,v_k)=gen(w_1,\ldots,w_t)=w_1R+w_2R+\cdots+w_tR,
\end{equation*} 
where the sum is direct. In this case we are done. 

Case 2: Suppose that the $j$-th column is the first column that has at least $2$ non-zero entries. Let $p$ be the number of non-zero entries it has, say $w_r^j\neq 0 \neq w_s^j$ with $r<s$ are the first two non-zero entries, (we necessarily have $r,s\leq j$) where $w_r^j$ is the $j$-th entry  of row $w_r$ and $w_s^j$  the $j$-th entry  of row $w_s.$
Let 
\begin{align}
(\alpha, \beta, \gamma )\in R^3  \thickspace \thickspace \mbox{such that}  \thickspace \thickspace (\alpha +\beta) \lambda \neq \alpha  \lambda  + \beta  \lambda.
\label{f}
\end{align}
We apply what we will call the \say{\textsl{distributivity trick}}:

Let $\alpha'= (w_r^j)^{-1}\alpha$ and $\beta'= (w_s^j)^{-1}\beta$. Then form a new row
\begin{equation*}
	\theta=(w_r \alpha' + w_s  \beta') \lambda - w_r  (\alpha' \lambda)-w_s (\beta ' \lambda). 
\end{equation*}
Since $\theta \in LC_2(w_r,w_s)$ we have $\theta\in gen(w_1,\ldots,w_t)$.

 For $ 1 \leq l < j,$ either $w_r^l$ or $w_s^l$ is zero because the $j$-th column is the first column that has two non-zero entries, thus $\theta^l=0$. Note that by the choice of $\alpha,\beta,\lambda$, we have 
\begin{align*}
\theta^j &=(w_r^j \alpha ' + w_s^j \beta') \lambda - (w_r^j\alpha ') \lambda- (w_s^j\beta ') \lambda \\
&=\big ( w_r^j (w_r^j)^{-1} \alpha + w_s^j(w_s^j)^{-1}  \beta \big ) \lambda- \big (  w_r^j ( w_r^j )^{-1} \alpha \big ) \lambda - \big (  w_s^j ( w_s^j )^{-1} \beta \big ) \lambda \\
&=(\alpha +\beta) \lambda - \alpha \lambda - \beta \lambda \neq 0.
\end{align*}
It follows  that $\theta^j\neq 0$.
Hence $ \theta =(0, \ldots, 0,\theta^j,\theta^{j+1},  \ldots, \theta^n )$. We now multiply the row $\theta$ by $ (\theta^j)^{-1},$ obtaining the row $\phi=(0, \ldots, 0,1,\theta^{j+1}(\theta^j)^{-1},  \ldots, \theta^n (\theta^j)^{-1}) \in gen(w_1,\ldots,w_k)$ where  $\phi ^j=1$ is  the pivot that we have created.

 As a next step, we form a new matrix of size $(t+1) \times m $ by adding $\phi$ to the rows $w_1, \ldots,w_t$ (just after the row $w_s$). In  this augmented matrix we replace the rows $w_r, w_s$  with $ y_r=w_r-(w_r^j) \phi, y_s=w_r-(w_s^j) \phi$, respectively. This yields another new matrix  composed of  the rows $w_1,\ldots,w_{r-1},y_r, \ldots, y_s, \phi, w_{s+1},  \ldots,w_t $ which has $p-1$ non-zero entries in the $j$-th column. By Lemma \ref{l1}, the  \textit{gen} of the rows of the augmented matrix is the  \textit{gen} of the rows of $W$ (which in turn is  $gen(v_1,\ldots,v_k)$).
Hence,
\begin{align*}
gen(v_1, \ldots,v_k) & = gen(w_1,\ldots,w_r, \ldots, w_s, \ldots,w_t)\\
&= gen(w_1,\ldots,y_r, \ldots, y_s, \phi, \ldots,w_t).
\end{align*}

We  keep repeating the \say{\textsl{distributivity trick}} on the $j$-th column until we form another new matrix which has only one non-zero entry at the $j$-th column.

By continuing this process, we can eliminate all columns with more than one non-zero entry. Let the final matrix have rows $u_1,u_2,\ldots, u_{k'}$. Then
\begin{equation*}
	gen(v_1,\ldots,v_k)  =gen(w_1,\ldots,w_t)=gen(u_1,\ldots,u_{k'})=u_1R+u_2R+\ldots+u_{k'}R, 
\end{equation*} 
where the sum is direct.

\end{proof}

The procedure described in the proof  of Theorem \ref{th2} is called the expanded Gaussian elimination (eGe algorithm). The focus of the next section is motivated by the \say{\textsl{distributivity trick}} used in the proof  of Theorem \ref{th2}.

\section{On the generalized  distributive set of a finite Dickson nearfield}


Let $(R,+, \circ)$ be a nearfield. We  use $D(R)$ to denote the set of all distributive elements of $R$, i.e.,
\begin{equation*}
D(R) = \lbrace \lambda \in R : \thickspace (\alpha+\beta) \circ \lambda = \alpha  \circ  \lambda+\beta \circ \lambda \thickspace \mbox{for all}\thickspace \alpha, \beta  \in R \rbrace.
\end{equation*} 
and $C(R)$ to denote  the set of elements that commute with every element of $R$, also called the multiplicative center of $ (R, \circ),$  i.e.,
\begin{align*}
C(R)= \left\lbrace x \in R : x\circ y=y \circ x  \thickspace  \mbox{for all}\thickspace y \in R \right\rbrace. 
\end{align*}

\begin{thm}(\cite{zemmer1964}) Let $R$ be a nearfield.  Then $D(R)$ under the operations of $R$ is a skewfield and is a subnearfield of $R$.
\label{tp}
\end{thm}
Note that it is known in \cite{aichinger2004multiplicative} that for any nearring $R$ the multiplicative center  of $R$ is not always a subnearring of $R$.  The authors in \cite{cannon2007centers} determined when  $C(R)$ is a subnearring of $R$.

\begin{defn}(\cite{cannon2007centers}) Let $R$ be a nearring  and $D(R)$ be the distributive elements of $R$. Then the generalized center of $R$ is the set
$$GC(R)=\{ x \in R:\thickspace  x \circ y=y \circ x  \thickspace  \mbox{for all}\thickspace  y\in D(R) \}$$
is called the set of elements of $R$ that commutes with elements of $D(R)$.
\end{defn}
\begin{lem}(\cite{cannon2007centers})   Let $R$ be a nearfield. Then $GC(R)$  under the operations of $R$ is a subnearfield of $R.$
\label{lem41}
\end{lem}

 Now consider $(R,+,\circ)$ as  a finite Dickson nearfield for the Dickson pair $(q,n)$ with $n>1.$ 

\begin{thm}(\cite{ellerskarzel1964}) Let $R$ be a finite Dickson nearfield that arises from the Dickson pair $(q,n).$ Then 
$D(R)=C(R) \cong \mathbb{F}_{q}$.
\label{thp}
\end{thm}
So the distributive elements of a finite Dickson nearfield $R$ under the operations of $R$ form a subnearfield of the nearfield and under the operations of the field $\mathbb{F}_{q^n}$ form a subfield of size $q$. Thus there are elements $(\alpha, \beta, \lambda) \in R^3$ such that $(\alpha+ \beta) \circ  \lambda \neq \alpha \circ \lambda + \beta \circ  \lambda. $ In the eGe algorithm described in the proof of Theorem \ref{th2}, the  \say{distributivity trick} requires that a triple of non-distributive elements (see equation (\ref{f})) be chosen at each  step and used in the creation of new rows. This leads us to investigate   the distributive elements of finite Dickson nearfields. 

\subsection{Some properties}Given $k  \in \{ 1, \ldots, n \},$  an  $H$-coset  shall refer to a coset of the form  $g^{[k]_q}H$. We have the following:

\begin{lem}
Let $R \in DN(q,n)$ where $(q,n)$ is a Dickson pair. Let $(\alpha, \beta) \in R^2.$  If  $\alpha, \beta, \alpha + \beta $ belong to the same $H$-coset, then 
$
(\alpha + \beta) \circ \lambda = \alpha \circ \lambda + \beta \circ \lambda$ for all $ \lambda \in R.$
\label{la}
\end{lem}
\begin{proof}Let $g$ be such that $\mathbb{F}_{q^n}^*= \big < g \big >$ and 
$H=\big < g^n \big >$. By the finite Dickson nearfield construction, the set of all $H$-cosets is presented as $
\mathbb{F}_{q^n}^*/ H = \big \lbrace  H,g^{[1]_q}H,\ldots, g^{[n]_q}  H\big \rbrace.$
 Assume that   $\alpha, \beta, \alpha + \beta \in g^{[k]_q}H$ for some $1 \leq k \leq n.$ 
Then 
\begin{align*}
(\alpha + \beta) \circ \lambda =(\alpha + \beta)  \lambda^{q^k} =\alpha \lambda^{q^k}  + \beta\lambda^{q^k}= \alpha \circ \lambda + \beta \circ \lambda,
\end{align*}
 for all $\lambda \in R.$
\end{proof}
The following is an immediate consequence of Lemma \ref{la}.
\begin{cor}
Let $R \in DN(q,n)$ and $(\alpha, \beta) \in R^2$ such that $\alpha, \beta, \alpha + \beta $ belong to the same $H$-coset. Then $C(D(\alpha, \beta))=D(R)$.
\label{corp}
\end{cor}
Now let us look at the case for  Dickson pairs $(q,n)$ where $n=2$. We have the following.
\begin{lem}
Let $(q,n)=(p^l,2)$ where $p$ is prime and $R \in DN(q,2)$. Let $(\alpha, \beta) \in R^2$ and  assume that $\alpha, \beta, \alpha + \beta $ don't belong all to the same $H$-cosets.  We have $(\alpha + \beta) \circ \lambda = \alpha \circ \lambda + \beta \circ  \lambda$ if and only if $\lambda \in D(R)$.
\label{thh}
\end{lem}
\begin{proof} Suppose $\alpha, \beta, \alpha + \beta$ belong to different $H$-cosets which means  $\alpha, \beta, \alpha + \beta $ are not all square and  not all non-square. We consider the case where $\alpha + \beta \in H$ and $\alpha, \beta \in gH$ (note that the other cases are similar).
If $(\alpha + \beta) \circ \lambda = \alpha \circ \lambda + \beta \circ  \lambda$  then we have $(\alpha+ \beta ) \lambda =\alpha \lambda^q +  \beta \lambda^q$. Thus $\lambda^{p^l}-\lambda=0$ and since every $\lambda \in \mathbb{F}_q$ is a solution of this equation, they are all the solutions. By Theorem \ref{thp} we have $\lambda \in D(R)$. The converse is straightforward.
\end{proof}

By Lemma \ref{thh} we deduce the following.
\begin{cor} Let $(q,2)$ be a Dickson pair with $q=p^l.$ Let $g$ be a generator of $\mathbb{F}_{q^2}^ *$ and  $R$  the finite nearfield constructed with $H= \big < g^2 \big > $.
For all pairs $(\alpha, \beta) \in R^2,$ $(\alpha + \beta) \circ \lambda = \alpha \circ \lambda + \beta \circ  \lambda$ if and only if either all $\alpha, \beta, \alpha + \beta$ belong to the same $H$-coset or $\lambda \in D(R).$
\label{cr}
\end{cor}

 Lemma \ref{thh} does not hold in   general for any finite Dickson nearfield $DN_g(q,n)$ where $(q,n)$ is a Dickson pair. The following example gives an illustration.
\begin{exa} Let $R \in DN(q,2)$ and $(\alpha, \beta) \in R^2$ where $\alpha, \beta, \alpha+\beta$ belong to different $H$-coset. Then by Lemma \ref{thh} and Corollary \ref{cr},  $(\alpha + \beta) \circ\lambda = \alpha \circ \lambda + \beta \circ  \lambda$  will always  lead to  the equation $\lambda ^{q}- \lambda=0$ and all the solutions will be in $\mathbb{F}_q.$
But  Lemma \ref{thh} can fail for any positive integer $n>2$. For instance given  $R = DN_g(5,4)$ and we take the irreducible polynomial $X^4+2$ of degree $4$ in $\mathbb{F}_5$ where $x$ is a root of $X^4+2 \in \mathbb{Z}_5[X]$. Let $g$ be such that  $\mathbb{F}_{5^4}^{*} = \big < g \big >$ and $H = \big < g^4 \big >$. The quotient group is represented by 
 \begin{align*}
\mathbb{F}_{5^4}^{ *} / H  =\big \lbrace  gH,  g^6H,g^{31}H,g^{156}H   \big \rbrace= \{H, gH,g^2H, g^3H \}.
\end{align*}
Let $\alpha,\beta \in  \mathbb{F}_{5^4},$ we have
\begin{align*}
\alpha \circ \beta  = 
\begin{cases}
 \alpha \cdot \beta   \thickspace  \text{if $\alpha \in H$}  \\
 \alpha \cdot \beta ^5  \thickspace    \text{if $ \alpha \in gH $} \\
 \alpha \cdot \beta ^{25}    \thickspace   \text{if $ \alpha \in g^2H $} \\ 
 \alpha \cdot \beta  ^{125}  \thickspace \text{if $ \alpha \in g^{3}H $}.
\end{cases} 
\end{align*}
 We consider $g=x+2$. Let $\alpha=3 \in H, \beta =x^2+2 \in H$. Then $\alpha + \beta \in g^2H.$ In fact $\lambda =x^2+1 \in g^2 H$ distributes over the pair $(\alpha, \beta)$. To see this, $(\alpha+ \beta) \circ \lambda= (3+x^2+2) \circ (x^2+1)=x^2+x^4.$ Also $\alpha \circ \lambda + \beta \circ \lambda= 3 \circ (x^2+1)+ (x^2+2) \circ (x^2+1)=x^4+x^2.$ But $\lambda \notin D(R)=\mathbb{F}_5.$ Note that $\lambda \notin D(R)$ but distributes over the pair $(\alpha,\beta).$ So $\lambda$ belong to a certain large distributive set of $R.$
 \label{rem}
\end{exa}
%
%
We now introduce the following generalized distributive set for a given pair in a nearfield.
\begin{defn}Let $R$ be a nearfield.
Given a pair $(\alpha, \beta) \in R^2,$ the generalized distributive set 
\begin{align*}
D(\alpha, \beta)= \big \{  \lambda \in R : (\alpha + \beta) \circ \lambda  = \alpha \circ \lambda + \beta \circ \lambda \big  \}
\end{align*} is the set of elements in $R$ that distribute over the pair $(\alpha, \beta).$
\end{defn}

Note that in the eGe algorithm, we have to choose at every  step of the creation of a new row a triple $(\alpha, \beta, \lambda) \in R^3$ such that $\lambda \notin D(\alpha, \beta)$ for the implementation of the \say{distributivity trick}. It is not difficult to see the following.
\begin{lem}Let $R$ be a nearfield. We have
\begin{align*}
\bigcap_{\alpha,\beta \in R}D(\alpha, \beta)=D(R).
\end{align*}
\end{lem}
\subsection{Some results on $D(\alpha,\beta)$ where $\alpha, \beta \in DN_g(q,n)$} 
We remind the reader that by definition a subset $S \subseteq \mathbb{F}_{p^n} $  that is a field is a subfield of the finite field $\mathbb{F}_{p^n}$. Also the subfields of $\mathbb{F}_{p^n}$ are the fields $\mathbb{F}_{p^k}$ where $k$ divides $n.$

Our aim is to determine $D(\alpha, \beta)$ where $(\alpha, \beta) \in R^2$  for $R$  a finite Dickson nearfield that arises from the Dickson pair $(q,n).$ Note that if $n=1$  then $DN(q,1)$ is the set of all finite fields of order $q$. Hence in this case $D(\alpha, \beta)= \mathbb{F}_{q}$ and also $C \big (D(\alpha, \beta) \big )=\mathbb{F}_{q}$. 

 From Lemma \ref{la}, Lemma \ref{thh} and Theorem \ref{tp}, we  deduce the following case  where $n=2$.
\begin{lem}
Let $(q,2)$ be a Dickson pair with $q=p^l$ for some prime $p$ and integer $l.$ Let $g$ be a generator of $\mathbb{F}_{q^2}^*$ and let $R$ be the nearfield constructed with $H = \big < g^2 \big >.$ Let $ (\alpha, \beta) \in R^2.$ Then
\begin{itemize}
\item[(i)]  $D(\alpha, \beta)$ is a subnearfield of $R$.
\item[(ii)] $C \big (D(\alpha, \beta) \big )$ is a subnearfield of $R$.
\end{itemize}

\end{lem}
\begin{proof} ~
\begin{itemize}
\item[(i)] Let $R \in DN(q,2)$ and $ \alpha, \beta \in R.$
We have $D(\alpha, \beta)=R$ or $D(\alpha, \beta)=D(R)$. So  $D(\alpha, \beta)$ is a subnearfield of $R.$
\item[(ii)] Suppose $\alpha, \beta, \alpha + \beta $ belong all to the same $H$-coset. Then by Corollary \ref{corp} we have $C \big (D(\alpha, \beta) \big )=C(R)=D(R)$ and by Theorem \ref{tp} $C \big (D(\alpha, \beta) \big )$ is a subnearfield of $R$. Furthermore, suppose $\alpha, \beta, \alpha + \beta $ don't belong all to the same $H$-coset. We have $D(\alpha, \beta)=D(R)$. Hence by Lemma \ref{lem41} 
\begin{align*}
C \big (D(\alpha, \beta) \big )= C \big (D(R) \big )= D(R) \cap GC(R)
\end{align*}
 is a subnearfield of $R$.

\end{itemize}

\end{proof}

Furthermore it is not difficult to see the following. Let $(q,n)$ be a Dickson pair with $q=p^l$ for some prime $p$ and integers $l,n.$ Let $g$ be a generator of $\mathbb{F}_{q^n} ^*$ and let $R$ be the nearfield constructed with $H = \big < g^n \big >.$ Then for each positive integer $h$ there exists a unique  subfield of $\mathbb{F}_{q^n}$  of order $p^h$   such that  $g^{\frac{q^n-1}{p^{h}-1}}$ generates its multiplicative group. In the next theorem we shall give a sufficient condition on $\alpha, \beta$ for which $D(\alpha, \beta)$ is a subfield of $\mathbb{F}_{q^n}.$ 
\begin{thm}
Let $(q,n)$ be a Dickson pair with $q=p^l$ for some prime $p$ and positive integers $l,n$ such that $n >2.$ Let $g$ be a generator of $\mathbb{F}_{q^n}^*$ and  $R$  the finite nearfield constructed with $H = \big < g^n \big >.$ Let $ \alpha, \beta \in R.$ If at least two of $\alpha, \beta, \alpha + \beta$ are in the same $H$-coset then $D(\alpha,\beta)$  is a subfield of  $\mathbb{F}_{q^{n}}$ of order $p^h$ for some $h$ dividing $l \cdot n.$
\label{thmm}
\end{thm}
\begin{proof} ~

Assume that  $\alpha, \beta, \alpha + \beta$ are all in the same $H$-coset. Then By Lemma \ref{la}, $(\alpha + \beta) \circ \lambda = \alpha \circ  \lambda+ \beta \circ \lambda$ for all $\lambda \in R.$ Hence $D(\alpha, \beta) $ coincides with  $\mathbb{F}_{q^n}$.

Assume now that exactly two of $\alpha, \beta, \alpha + \beta$ are in the same $H$-coset. We consider the case where   $ \alpha, \beta \in g^{[s]_q}H$ and $\alpha + \beta \in  g^{[t]_q}H $ for $s \neq t$ (note that the other cases are similar). 
Then $
(\alpha + \beta ) \circ \lambda = (\alpha + \beta) \lambda^{q^t}= \alpha \lambda ^{q^t} + \beta \lambda ^{q^t}$.  Also
$
\alpha \circ \lambda + \beta \circ \lambda = \alpha \lambda^{q^s}+ \beta \lambda^{q^s}$.
Hence
 \begin{align*}
 (\alpha + \beta ) \circ \lambda =\alpha \circ \lambda + \beta \circ \lambda  \Rightarrow (\alpha + \beta) (\lambda ^{q^t}- \lambda ^{q^s})=0
 \end{align*}
 and then $\lambda =0 $ is solution of the equation. Now suppose $\lambda \neq 0,$  so
 \begin{align*}
 (\alpha + \beta) (\lambda ^{q^t}- \lambda ^{q^s})=0 \Rightarrow \lambda ^{q^t-q^s}-1=0.
\end{align*} 
It follows that
\begin{align*}
( \lambda ^{q^t-q^s}-1 \big )^q=0 \Rightarrow  \lambda ^{q^{t+1}-q^{s+1}}-1=0.
\end{align*}
Continuing this procedure (raising to the power $q^\epsilon$ such that $n= s +\epsilon$) we obtain 
\begin{align*}
\lambda ^{q^{t+ \epsilon}-q^{s + \epsilon}}-1=0.
\end{align*}
 Hence, we  have $\lambda ^{q^r-q^n}-1=0$ where $r=t+ \epsilon$ and since $ \lambda ^{q^n}= \lambda$, we get $\lambda^{q^r}-\lambda=0.$  We know that   $q=p^l$ where $p$ is a prime number and $l,n$ are positive integers. Hence we get the following equation 
 \begin{align*}
 (\Sigma):\lambda^{p^{k}}-\lambda=0 \thickspace \mbox{where}  \thickspace \lambda \in \mathbb{F}_{p^{m}}
 \end{align*} for some positive integers $k=l \cdot r$ and $m=l \cdot n$.
\begin{itemize}
\item[(i)]  Suppose $k$ divides $m$. Then there exists exactly one subfield  of  $\mathbb{F}_{p^m} $ which is isomorphic to $\mathbb{F}_{p^k}.$ So all the solutions of   $(\Sigma)$ are in $\mathbb{F}_{p^m}$. Hence  $D(\alpha, \beta) $ coincides with $\mathbb{F}_{p^{k}}$ which is  a subfield of $\mathbb{F}_{p^{m}}$. 
\item[(ii)] 
Suppose $k$ does not divide $m$. Note that $\lambda=0$ is a solution of $(\Sigma)$. Set $\delta =\textsc{gcd}(m,k)$.  Now let $\lambda \in \mathbb{F}_{p^m}^*$ be a solution of $(\Sigma)$. So $\lambda=g^a$ for some  $ 0 \leq a < p^m-1$ and $g^{a(p^k)}-g^a=0.$ Then $g^{a(p^k-1)}=1$. We have,
\begin{align*}
(p^m -1) \thickspace \mbox{divides} \thickspace a (p^k-1).
\end{align*}
 So there exists an integer $t$ such that $a(p^k-1)=t(p^m-1).$ Since $\delta $ divides $m $  there exists $\theta \in \mathbb{N}$ such that $m= \delta \theta.$ Also since $\delta$ divides $ k $  there exists $\theta' \in \mathbb{N}$ such that $k= \delta \theta'.$ So we have $ p^m-1 = (p^{\delta}-1) \big (  (p^{\delta}) ^{\theta-1} + \ldots + p^{\delta}+ 1 \big )$ and $ p^k-1 = (p^{\delta}-1) \big (  (p^{\delta}) ^{\theta'-1} + \ldots + p^{\delta}+ 1 \big ).$  Furthermore, since  
\begin{align*}
\textsc{gcd} (p^m-1, p^k-1)=p^{\delta}-1,
\end{align*} 
 by Bezout's identity there exists some  integers $u$ and $v$ such that 
\begin{align*}
u(p^m-1)+ v(p^k-1)=p^{\delta}-1.
\end{align*} 
Hence substituting $a(p^k-1)=t(p^m-1)$ we get
 \begin{align*}
 au(p^m-1)+ vt(p^m-1)=a(p^{\delta}-1).
\end{align*} 
Thus $(p^m-1) $ divides $a(p^{\delta}-1).$ It follows that $ \frac{p^m-1}{p^\delta-1} $ divides  $a$. So  $ a = \frac{p^m-1}{p^\delta-1} b$ where $b \in \mathbb{N}.$ Now, 
\begin{align*}
0 \leq a < p^m-1 \Leftrightarrow  0 \leq b < p^{\delta}-1.
\end{align*} 
 Reciprocally,  let $ 0 \leq b < p^\delta-1$ and  $\lambda_0=g^a.$  We know that  $a = \frac{p^m-1}{p^\delta-1}b$  and $p^k-1= t'(p^\delta-1)$  for some  integer $t'$. We have, 
 \begin{align*}
 \lambda_0^{p^k-1}-1 &= (g^a)^{t'(p^{\delta}-1)}-1 \\
  & = g^{(p^m-1)t'b}-1 \\
  &= 1^{t'b}-1 =0.
 \end{align*} It follows that $\lambda_0$ is a solution of the equation $\lambda^{p^k}-\lambda=0$. Hence the set of solutions (denoted as $S(\Sigma)$) of the equation  $(\Sigma)$  for $\lambda \in \mathbb{F}_{p^{m}}$ are presented as
\begin{align*}
S(\Sigma)= \left \{ 0 \right \} \cup \left \{  g^{ \frac{p^m-1}{p^{\delta}-1}b}: \thickspace 0 \leq b  < p^{\delta}-1 \right \}.
\end{align*} 
The elements of $S(\Sigma)$ are all distinct since $g^{ \frac{p^m-1}{p^{\delta}-1}}$ is a generator of the multiplicative group of the subfield $\mathbb{F}_{p^{\delta}}$ of  $\mathbb{F}_{p^m}$ (since $\delta$ divides $m$). Hence  
\begin{align*}
\vert S(\Sigma) \vert= p^{\delta}-1+1=p^{\delta} =p^{l  \cdot \textsc{gcd}(t+n-s,n)}.
\end{align*}
 So all the solutions of $(\Sigma)$ are in the  finite field of order $p^\delta$. Hence $D(\alpha,\beta)$ coincides with $S(\Sigma)  = \mathbb{F}_{p^ \delta}$. 
\end{itemize}

\end{proof}

\begin{exa}  Suppose  $R = DN_g(5,4)$ with the generator $g=x+2$ (see Example \ref{rem} regarding the construction of $DN_g(5,4)$). Let $\alpha, \beta \in R$ such that $\alpha= x+2 \in gH$ and $\beta =x^3+x^2+2x+3 \in g^2H$ where $x$ is the root of the irreducible polynomial $X^4+2 \in \mathbb{Z}_5[X]$. Then $\alpha +\beta \in gH.$ There exist $\lambda_1= x^2+3, \thickspace \lambda_2= 3x^2+2 \in D(\alpha, \beta)$ such that $ \lambda_1 \in H$ and $ \lambda_2 \in g^2H.$   So the distributive elements $\lambda_1$ and $\lambda_2$ don't belong necessarily to the same $H$-coset.
\end{exa}

Note that if $R \in DN(q,n)$ and $(\alpha, \beta) \in R^2$ then it is not difficult to see that $D(R) \subseteq D(\alpha, \beta) \cap H.$ Moreover we  deduce the following:
\begin{lem} Let $(q,n)$ be a Dickson pair with $q=p^l$ for some prime $p$ and integers $l,n.$ Let $g$ be a generator of $\mathbb{F}_{q^n} ^*$ and let $R$ be the nearfield constructed with $H = \big < g^n \big >.$
Let $(\alpha, \beta, \alpha+ \beta) \in g^{[r]_q}H \times g^{[s]_q}H \times H $ where $s \neq r$  and $r,s$ divide $n$. Suppose  that $n$  divides $\frac{q^n-1}{q^r-1} $ or $n$ divides $ \frac{q^n-1}{q^s-1} $.  Let $S$ be the subset of $\mathbb{F}_{q^n}$ defined by
\begin{align*}
S= \left \{ g^m : q^n-1  \thickspace \mbox{divides} \thickspace m(q^r-1) \thickspace \mbox{and} \thickspace q^n-1  \thickspace \mbox{divides} \thickspace m(q^s-1) \thickspace \mbox{where} \thickspace 0 \leq m < q^n-1 \right \}.
\end{align*}  Then $S \subseteq  D(\alpha, \beta) \cap H$.  
\end{lem}
\begin{proof} Suppose $(\alpha, \beta, \alpha+ \beta) \in g^{[r]_q}H \times g^{[s]_q}H \times H $  where $s \neq r$  and $r,s$ divide $n$.  Let $\lambda \in D(\alpha, \beta).$ So $\lambda \in \mathbb{F}_{q^n}$ and $\alpha (\lambda - \lambda^{q^r})= \beta( \lambda^{q^s} - \lambda)$. Let $\lambda \in \mathbb{F}_{q^n}^*$, then $\lambda =g^m$ for some $0 \leq m < q^n-1.$ In fact the elements of $S$ are solutions of $\alpha (\lambda - \lambda^{q^r})= \beta( \lambda^{q^s} - \lambda)$.
To see this, we have
 \begin{align*}
\alpha (g^m-g^{mq^r}) -\beta (g^{mq^s}-g^m)= \alpha g^m(1-g^{m(q^r-1)})- \beta g^m(1-g^{m(q^s-1)}).
\end{align*}
Since $q^n-1$ divides $m(q^r-1)$ then there exists a positive integer $u$ such that $m(q^r-1)=u(q^n-1)$. Also since $q^n-1$ divides $m(q^s-1)$ then there exists a positive integer $v$ such that $m(q^s-1)=v(q^n-1)$. So we have,
\begin{align*}
\alpha g^m(1-g^{m(q^r-1)})- \beta g^m(1-g^{m(q^s-1)}) & =\alpha g^m(1-(g^{q^n-1})^u)- \beta g^m(1-(g^{q^n-1})^v) \\
&= \alpha g^m(1-1) - \beta g^m (1-1)=0.
\end{align*}
So $S \subseteq D(\alpha, \beta)$. Furthermore,
 \begin{align*}
S & = \left \{ g^m : \frac{q^n-1}{q^r-1}  \thickspace \mbox{divides} \thickspace m \thickspace \mbox{and} \thickspace \frac{q^n-1}{q^s-1}  \thickspace \mbox{divides} \thickspace m  \right \} \\
&=\left \{ g^m : m= a  \big ( \textsc{lcm} (  \frac{q^n-1}{q^r-1}, \frac{q^n-1}{q^s-1}) \big )  \thickspace \mbox{for} \thickspace a \in \mathbb{N} \right \}  \\
&=\left \{  g^{a  \big ( \textsc{lcm} (\frac{q^n-1}{q^r-1},\frac{q^n-1}{q^s-1}) \big )}  :  a  \in \mathbb{N} \right \}.
\end{align*}
 Since $n$ divides $\frac{q^n-1}{q^r-1}$ or  $n$ divides $ \frac{q^n-1}{q^s-1} $ then there exists a positive integer $b$ such that  $\textsc{lcm}(\frac{q^n-1}{q^r-1},\frac{q^n-1}{q^s-1})=n b$. Thus $S=\{ (g^n)^{ab}: (a,b) \in \mathbb{N}^2\} \subseteq H.$

\end{proof}

We now look at $D(\alpha, \beta)$  where $\alpha, \beta, \alpha + \beta$ are all in distinct $H$-cosets. We deduce the following construction of subfield of $\mathbb{F}_{q^n}$ by making use of $D(\alpha, \beta)$.
\begin{thm}
Let $(q,n)$ be a Dickson pair with $q=p^l$ for some prime $p$ and integers $l,n$  such that $n >2.$  Let $g$ be a generator of $\mathbb{F}_{q^n}^*$ and  $R$  the finite  nearfield constructed with $H = \big < g^n \big >.$ Suppose $(r,s,t) \in \mathbb{N}^3$ such that $0 < t< s< r \leq n$ where  $r-s $ divides $   n$ and $r-t$ divides $ n$.  Let  $ (\alpha, \beta , \alpha + \beta )\in g^{[r]_q}H \times g^{[s]_q}H \times g^{[t]_q}H$. Let $F_{r,s,t}(\alpha, \beta)$ be the subset   of  $\mathbb{F}_{q^n}$ defined by
\begin{align*}
F_{r,s,t}(\alpha, \beta)= D(\alpha, \beta) \cap \mathbb{F}_{q^{r-s}} \cap  \mathbb{F}_{q^{r-t}}.
\end{align*}
 Then $F_{r,s,t}(\alpha, \beta)$ is a subfield of  $\mathbb{F}_{q^n}.$ 
\label{pr}
\end{thm}
\begin{proof} Let  $ (\alpha, \beta , \alpha + \beta )\in g^{[r]_q}H \times g^{[s]_q}H \times g^{[t]_q}H$ such that $0 < t< s< r \leq n.$ Let $\lambda \in D(\alpha, \beta)$, so we have 
$\alpha (\lambda^{q^t}-\lambda^{q^r})= \beta (\lambda^{q^s}-\lambda^{q^t}).$ Suppose $r-s$ divides $  n$ and $r-t $ divides $  n $. Then  $\mathbb{F}_{q^{r-s}} \cap  \mathbb{F}_{q^{r-t}}$ is a subfield of  $\mathbb{F}_{q^n}.$ 
Note that $ \lambda \in \mathbb{F}_{q^{r-s}} \cap \mathbb{F}_{q^{r-t}} \Leftrightarrow$  $\lambda^{q^{r-s}}= \lambda$ and $\lambda^{q^{r-t}}= \lambda.$  It follows that  $(\lambda^{q^{r-s}}-\lambda) ^ {q^s}=0 $ and  $(\lambda^{q^{r-t}}-\lambda) ^ {q^t}=0$. Thus   $\lambda^{q^r}= \lambda^{q^s} $ and  $ \lambda^{q^r}=\lambda^{q^t}$. Also $ F_{r,s,t}(\alpha, \beta) \neq \emptyset$ since $0,1 \in F_{r,s,t}(\alpha, \beta).$ Suppose $\lambda_1, \lambda_2 \in F_{r,s,t}(\alpha, \beta) .$ We have 
\begin{align*}
 \alpha \big ( (\lambda_1 + \lambda_2)^ {q^t} - (\lambda_1 + \lambda_2)^ {q^r}\big ) &= \alpha( \lambda_1^{q^t}+\lambda_2^{q^t} -\lambda_1^{q^r}-\lambda_2^{q^r}) \\
 &= \alpha (\lambda_1^{q^t}-\lambda_1^{q^r}) + \alpha (\lambda_2^{q^t}-\lambda_2^{q^r}) \\
 &= \beta (\lambda_1^{q^s}-\lambda_1^{q^t}) + \beta (\lambda_2^{q^s}-\lambda_2^{q^t}) \\
 &= \beta \big ( (\lambda_1 + \lambda_2)^ {q^s} - (\lambda_1 + \lambda_2)^ {q^t}\big ).
\end{align*}
It follows that $\lambda_1 + \lambda_2  \in D(\alpha, \beta).$ Also $\lambda_1 + \lambda_2 \in \mathbb{F}_{q^{r-s}} \cap  \mathbb{F}_{q^{r-t}}.$
Hence $ \lambda_1 + \lambda_2 \in F_{r,s,t}(\alpha, \beta).$ Furthermore,
\begin{align*}
\alpha \big ( (\lambda_1 \lambda_2)^ {q^t} - (\lambda_1 \lambda_2)^ {q^r}\big ) &= \alpha( \lambda_1^{q^t}\lambda_2^{q^t} -\lambda_1^{q^r}\lambda_2^{q^r}) \\
&= \alpha (\lambda_1^{q^r}\lambda_2^{q^t} -\lambda_1^{q^r}\lambda_2^{q^r}) \\
&=\alpha \lambda_1^{q^r} ( \lambda_2^{q^t}- \lambda_2^{q^r}) \\
&= \lambda_1^{q^r} \beta  ( \lambda_2^{q^s}- \lambda_2^{q^t}) \\
&= \beta ( \lambda_1^{q^r}\lambda_2^{q^s}- \lambda_1^{q^r} \lambda_2^{q^t}) \\
&= \beta ( \lambda_1^{q^s}\lambda_2^{q^s}- \lambda_1^{q^t} \lambda_2^{q^t}) \\
&= \beta \big ( (\lambda_1 \lambda_2)^ {q^s} - (\lambda_1 \lambda_2)^ {q^t}\big ). 
\end{align*}
It follows that $\lambda_1  \lambda_2  \in D(\alpha, \beta).$ Also $\lambda_1  \lambda_2 \in \mathbb{F}_{q^{r-s}} \cap  \mathbb{F}_{q^{r-t}}.$ Hence $ \lambda_1  \lambda_2 \in F_{r,s,t}(\alpha, \beta).$ We also have $\lambda_1 - \lambda_2 \in F_{r,s,t}(\alpha, \beta)$. So $ \big ( H_{r,s,t}(\alpha, \beta), +\big )$ is a subgroup of $(\mathbb{F}_{q^n},+).$ Suppose $\lambda \in \mathbb{F}_{q^{r-s}} \cap \mathbb{F}_{q^{r-t}} $. We know that if $\lambda \in D(\alpha,\beta) $ then  
\begin{align*}
\alpha ( \lambda ^{q^t}-\lambda ^{q^r} ) = \beta ( \lambda ^{q^s}-\lambda ^{q^t} ) \Leftrightarrow \alpha  ( \lambda ^ {q^r-q^t}-1)=- \lambda^{-q^t} \beta ( \lambda ^{q^s}-\lambda ^{q^t}).
\end{align*} Thus 
\begin{align*}
\alpha (\lambda ^ {-q^t}-\lambda ^ {-q^r})& =\alpha  \lambda ^ {-q^r} (\lambda ^ {q^r-q^t}-1) \\
&= \alpha \lambda^{-q^s} ( \lambda ^ {q^r-q^t}-1)\\
&=- \lambda^{-q^s} \lambda^{-q^t} \beta ( \lambda ^{q^s}-\lambda ^{q^t}) \\
&= \beta (\lambda ^ {-q^s}-\lambda ^ {-q^t}).
\end{align*} 
Thus $\lambda^{-1} \in D(\alpha,\beta). $ Hence $ \lambda_1 \lambda_2^{-1} \in F_{r,s,t}(\alpha, \beta).$ So $ \big ( F_{r,s,t}(\alpha, \beta), \cdot  \big )$ is a subgroup of $(\mathbb{F}_{q^n}^*,\cdot).$
\end{proof}

\begin{rem}
Only the case where $ 0 < t< s< r \leq n$ is  considered in the Theorem \ref{pr}. However the other cases can be deduced  in the similar way.
\end{rem}

\subsection{$D(\alpha,\beta)$ presented as a vector space where $\alpha, \beta \in DN_g(q,n)$} 
 
In this subsection it is shown that  $D(\alpha,\beta)$ has a vector space structure.  We have the following:
\begin{lem}
Let $(q,n)$ be a Dickson pair with $q=p^l$ for some prime $p$ and integers $l,n.$ Let $g$ be a generator of $\mathbb{F}_{q^n} ^ *$ and  $R$  the finite nearfield constructed with $H = \big < g^n \big >.$ Let  $ (\alpha, \beta )\in R^2.$ Then $D(\alpha,\beta)$ is an $\mathbb{F}$-vector space for some finite field $\mathbb{F}$.
\label{lprr}
\end{lem}
\begin{proof}We have:
\begin{itemize}
 \item Suppose that  $\alpha, \beta, \alpha + \beta$ are all in the same $H$-coset. Then $D(\alpha,\beta)=\mathbb{F}_{q^n}. $ Thus $D(\alpha,\beta)$ is an $\mathbb{F}$-vector space where $\mathbb{F}=\mathbb{F}_{q^n}$.

 \item Suppose that exactly two of  $\alpha, \beta, \alpha + \beta$ belong to the same $H$-coset. We consider the case where   $ \alpha, \beta \in g^{[s]_q}H$ and $\alpha + \beta \in g^{[t]_q}H $ for $s \neq t$ (note that the other cases are similar).  By the proof of Theorem \ref{thmm},  we have $D(\alpha,\beta)=\mathbb{F}_{q^{\textsc{gcd}(t+n-s,n)}}$. Thus   $D(\alpha,\beta)$ is an $\mathbb{F}$-vector space where $\mathbb{F}=\mathbb{F}_{q^{\textsc{gcd}(t+n-s,n)}}$.

 \item  Suppose that $\alpha \in g^{[r]_q}H$ and $\beta \in g^{[s]_q}H$ and $\alpha +\beta \in g^{[t]_q}H$ where $r,s$ and $t$ are all distinct. We define
\begin{align*}
\mathbb{F}= \{ k \in \mathbb{F}_{q^n}: k^{{q^t}}=k^{{q^r}}=k^{q^s}=k \}= \mathbb{F}_{q^t} \cap \mathbb{F}_{q^r} \cap \mathbb{F}_{q^s}.
\end{align*}

 Let $\lambda_1, \lambda_2 \in D(\alpha, \beta)$. Using the Frobenius identity, we have $\lambda_1 + \lambda_2 \in D(\alpha, \beta). $  Let $ k \in\mathbb{F} $ and $\lambda \in D(\alpha, \beta).$ We have,
\begin{align*}
\alpha \big (  (k \lambda )^{q^t}+ (k \lambda )^{q^r} \big ) &= \alpha k ^{q^t} (\lambda ^{q^t}- \lambda ^{q^r} ) \\
&= \beta k ^{q^t} (\lambda ^{q^s}- \lambda ^{q^t} )\\
&= \beta  (k ^{q^t} \lambda ^{q^s}- k ^{q^t} \lambda ^{q^t} )\\
&=\beta  \big (  (k \lambda) ^ {q^s}  - (k \lambda) ^ {q^t} \big ).
\end{align*} It follows  that $k \lambda \in D(\alpha, \beta).$ Thus $D(\alpha,\beta)$ is an $\mathbb{F}$-vector space where $\mathbb{F}=\mathbb{F}_{q^t} \cap \mathbb{F}_{q^r} \cap \mathbb{F}_{q^s}$.
\end{itemize}

\end{proof}

\begin{rem}
It is not difficult  to see that  $D(\alpha, \beta)$ is also an $\mathbb{F}_q$-vector space. 
\end{rem}
\begin{lem}
Let $(q,n)$ be a Dickson pair with $q=p^l$ for some prime $p$ and integers $l,n$  such that $n >2.$  Let $g$ be a generator of $\mathbb{F}_{q^n} ^*$ and  $R$ the nearfield constructed with $H = \big < g^n \big >.$ Let  $ (\alpha, \beta ) \in R^2$.  Then there exists a positive integer $k$ such that $\vert D(\alpha, \beta) \vert =q^k.$ 
\end{lem}
\begin{proof}
$D(\alpha, \beta)$ is a finite dimensional vector space over $\mathbb{F}_q$. Then $D(\alpha, \beta)$ has a basis over $\mathbb{F}_q$ consisting of say $k$ elements. Thus  $D(\alpha, \beta)$ has exactly $q^k$ elements.
\end{proof}
Our next goal is to find a basis for  $D(\alpha, \beta)$.
\paragraph{Computational aspect \\ \\} 

 Let $\alpha, \beta, \alpha +\beta \in DN_g(q,n)$ such that $\alpha \in g^{[r]_q}H,$  $\beta \in g^{[s]_q}H$ and $\alpha +\beta \in g^{[t]_q}H$ where $r,s$ and $t$ all different. We have
  \begin{align*}
(\alpha + \beta) \circ \lambda = \alpha \circ \lambda + \beta \circ \lambda \Leftrightarrow \alpha (\lambda^{q^t}-\lambda^{q^r})= \beta (\lambda^{q^s}-\lambda^{q^t})
\end{align*} for all $\lambda \in \mathbb{F}_{q^n}. $ We resort to computational methods  in order to check the nature of the distributive elements $D(\alpha,\beta)$. We derive  an algorithm called "DSS" (Distributive Set  Subfields) which tests whether $D(\alpha, \beta)$ is a finite field. 
 
 We construct a function $\phi$ from $\mathbb{F}_{q^n}$ to itself that maps $\lambda$ to $  \alpha(\lambda^{q^t}-\lambda^{q^r}) - \beta(\lambda^{q^s}-\lambda^{q^t})$. This function will be zero if and only if $\lambda$ is a solution to the equation $ \alpha(\lambda^{q^t}-\lambda^{q^r}) = \beta(\lambda^{q^s}-\lambda^{q^t})$. Moreover, $\phi $ is an $\mathbb{F}_q$-linear map, so it suffices to compute it on a basis. A basis of $\mathbb{F}_{q^n}  $ is $\{1, g, g^2, \ldots, g^{n-1} \}.$  The function $\phi$ is entirely determined by the vectors $\phi(1), \phi(g), \ldots, \phi(g^{n-1}).$ We can represent each vector $\phi(g^j)$ for $0 \leq j \leq n-1$ as 
 \begin{align*}
 \phi(g^j)=m_0^j1+m_1^jg+ \ldots+ m_{(n-1)}^jg^{n-1} \thickspace \mbox{where} \thickspace m_i^j \in \mathbb{F}_q \thickspace \mbox{with} \thickspace 0 \leq i \leq n-1 .
 \end{align*}
 Thus $\phi$ is entirely determined by the values of the matrix of size $n \times n$ defined by 
 $M=(m_i^j)_{\substack{ 0 \leq i \leq n-1 \\ 0 \leq j \leq n-1}}.$  Let $\lambda \in \mathbb{F}_{q^n},$  there exist $\lambda_0, \ldots, \lambda_{n-1} \in \mathbb{F}_q$ such that $\lambda=\lambda_01+ \ldots+ \lambda_{n-1}g^{n-1}.$ So $\lambda$ can be represented by $\begin{bmatrix}
\lambda_0 \\
\vdots \\
\lambda_{n-1}
\end{bmatrix}$. Thus for all $\lambda \in \mathbb{F}_{q^n}$ we have $\phi (\lambda)=M \begin{bmatrix}
\lambda_0 \\
\vdots \\
\lambda_{n-1}
\end{bmatrix}$ where $\lambda=\lambda_01+ \ldots+ \lambda_{n-1}g^{n-1}$. We have
 \begin{align*}
 D(\alpha, \beta) &= ker \phi \\
 &= Nullspace(M) \\
 &=\left \{\lambda \in \mathbb{F}_{q^n}: M \begin{bmatrix}
\lambda_0 \\
\vdots \\
\lambda_{n-1}
\end{bmatrix} =0 \right \}.
 \end{align*}
 We give the details of the full algorithm in the appendix (see Algorithm \ref{algo}). After the implementation  of the DSS algorithm, we have the following examples.
\begin{exa}~
\begin{enumerate}
\item Suppose that  $(q,n)=(4,3)$. Let $g$ be such that $\mathbb{F}_{4^3}^* = \big <g \big >$  and $H=\big < g^3 \big >.$   We know that by Theorem \ref{thmm}, if at least two of  $\alpha, \beta, \alpha + \beta$ belong to the same $H$-coset then  $D(\alpha, \beta)$  is a subfield of $\mathbb{F}_{q^n}$. Furthermore for all $\alpha \in H, \beta \in gH$ and $ (\alpha + \beta) \in g^2H,$  $D(\alpha, \beta)$ is a finite field. 
\item Suppose that  $(q,n)=(5,4)$. Then $D(\alpha, \beta)$ is also a finite field for every $\alpha, \beta \in DN_g(5,4)$ where $\mathbb{F}_{5^4}^* = \big <g \big >$  and $H=\big < g^4\big >.$ 
\item Suppose that $(q,n)=(7,9)$. Let $g$ be such that $\mathbb{F}_{7^9}^* =\big < g \big >$ and $H =\big < g^9 \big >.$ We know that $D(\alpha, \beta)$ is an $\mathbb{F}_7$-vector space. Also $\mathbb{F}_{7^9}$ is an $\mathbb{F}_7$-vector space with a basis  $\{1, g^2,\ldots, g^8 \}$. We take  some  elements   
\begin{align*}
 & \alpha=  4g^8 + 5g^7 + 3g^4 + 6g^3 + 6g^2 + 6g + 4 \in  \mathbb{F}_{7^9}^*  \\
& \beta =  5g^8 + 3g^7 + g^6 + 3g^5 + 3g^4 + g^3 + 3g^2 + 6g + 6 \in \mathbb{F}_{7^9}^*
\end{align*}
such that $\alpha, \beta, \alpha + \beta$ are in distinct $H$-cosets. Then it turns out that a basis of $D(\alpha,\beta)$ is $\{1,  g+2g^2+6g^5+6g^6+5g^7+6g^8\}$  and has dimension $2$.  Then
\begin{align*}
D(\alpha, \beta) =  \left \{ \sum_{i=1}^2 v_i \alpha_i \thickspace \vert \thickspace \alpha_i \in \mathbb{F}_7 \right \}
\end{align*}  where $v_1=1$ and $v_2= g+2g^2+6g^5+6g^6+5g^7+6g^8$. Hence $\vert D(\alpha,\beta) \vert=7^2$ and  $2$ does not divide $9$. Thus $D(\alpha,\beta)$ is not a finite field.
\item We also consider the Dickson pair $(q,n)=(5,8)$. Let $g$ be such that $\mathbb{F}_{5^8}^* =\big < g \big >$ and $H =\big < g^8 \big >.$ We know that $D(\alpha, \beta)$ is an $\mathbb{F}_5$-vector space. Also $\mathbb{F}_{5^8}$ is an $\mathbb{F}_5$-vector space with a basis  $\{1, g^2,\ldots, g^{7} \}$. We take  a pair of elements in $\mathbb{F}_{5^{8}}^*$ 
\begin{align*}
& \alpha =3g^6+4g^4+3g^3+3g^2+2g+2 \\
& \beta=4g^7+g^6+g^5+3g^4+4g^2+3g+2
\end{align*}
 such that $\alpha, \beta, \alpha + \beta$ are in distinct $H$-cosets. Then
\begin{align*}
D(\alpha, \beta) = \left \{ \sum_{i=1}^2 v_i \alpha_i \thickspace \vert \thickspace \alpha_i \in \mathbb{F}_5 \right \}
\end{align*}
  where $v_1=1$ and $v_2= 2g^7+3g^6+g$ and  $D(\alpha,\beta)$ has  dimension $2$. Hence $\vert D(\alpha,\beta) \vert=5^2$. Note that $2$ divides $8$. But $D(\alpha,\beta)$ is not closed under  the finite field multiplication. To see this,  
\begin{align*}
v_2^2=v_2 \cdot v_2= 4g^7+g^6+2g^5+g^4+2g^3+3g^2+2g+4.
\end{align*}  
   In fact 
\begin{align*}
  v_2^2 \notin D(\alpha,\beta).
\end{align*}   
  Hence $D(\alpha,\beta)$ is not a finite field. 
\end{enumerate}
\label{example1}
\end{exa}

The Example \ref{example1} leads us to deduce that if    $\alpha, \beta, \alpha + \beta$  belong to distinct $H$-cosets  then $ D(\alpha, \beta)$ is not in general a  subfield of the finite field $\mathbb{F}_{q^n}$.

During $1971$ and $1972$, Dancs  showed in \cite{susans1971,susans1972} that the subnearfield structure of a finite nearfield is analogous to the subfield structure of finite fields.
\begin{thm}(\cite{susans1971, susans1972})
Let $R$ be a finite nearfield of order $q^{n},$ where $q=p^l$ for some prime $p$.
\begin{enumerate}
\item[(i)] If $K$ is a subnearfield of $R$, then $\vert K \vert =p^h$ with $h$ dividing $l \cdot n$.
\item[(ii)] Conversely, if $h$ divides $l \cdot n$, then $R$ has a unique subnearfield $K$ of order $p^h.$
\end{enumerate} 
\label{t}
\end{thm}
Let $R \in DN(q,n)$ where $n>2$. Let $(\alpha, \beta) \in R^2$. Looking   at $D(\alpha, \beta)$  where $\alpha, \beta, \alpha + \beta$ are all in distinct $H$-cosets, assume that  $ (\alpha, \beta , \alpha + \beta )\in g^{[r]_q}H \times g^{[s]_q}H \times g^{[t]_q}H$ such that $r,s,t$ are all distinct.
Then
\begin{align*}
D(\alpha, \beta)= \left \{ \lambda \in R \thickspace : \thickspace \alpha (\lambda^{q^t}-\lambda^{q^r})= \beta (\lambda^{q^s}-\lambda^{q^t}) \right \}.
\end{align*}
 By Example \ref{example1} and appealing to Theorem \ref{t}, if we take the Dickson pair $(7,9)$ we see that there exist some pairs $(\alpha,\beta)$ such that $D(\alpha,\beta)$ is not a subnearfield of $DN_g(7,9).$

 Furthermore by proof of Theorem \ref{thmm} the following is an immediate consequence:

\begin{cor}
Let $(q,n)$ be a Dickson pair with $q=p^l$ for some prime $p$ and positive integers $l,n$ such that $n >2.$ Let $g$ be a generator of $\mathbb{F}_{q^n}^*$ and let $R$ be the nearfield constructed with $H = \big < g^n \big >.$ Let $(\alpha, \beta) \in R^2$. We have the following:

\begin{enumerate}
\item[(i)] If all $\alpha, \beta, \alpha +\beta$ belong to the same $H$-coset then $D(\alpha,\beta)=R$ and so is also a subnearfield of $R.$ 
\item[(ii)] If two of $\alpha, \beta , \alpha +\beta $ belong to $g^{[s]_q}H$  and the  third to $g^{[t]_q}H$ such that $\textsc{gcd}(t+n-s,n)=1$ where  $s \neq t$ then by Theorem \ref{tp} $D(\alpha, \beta)$ is  a subnearfield of $R.$ Furthermore $C \big (D(\alpha, \beta ) \big )=C \big (D(R) \big )= D(R) \cap GC(R).$
\end{enumerate}
\end{cor}
\begin{rem}
Let $(q=p^l,n)$ be a Dickson pair and  $R = DN_g(q,n)$ where $g$ is  a generator of $\mathbb{F}_{q^n} ^ *$ and  $R$  the finite nearfield constructed with $H = \big < g^n \big >.$  Every $\mathbb{F}_{p^k}$ with $k$ dividing $l$ is a subnearfield of $DN_g(q,n)$. To see this, by Theorem \ref{tp} $D(R)$ is a subnearfield of $R$. Also by Theorem \ref{thp} $D(R)$ is a finite field, so isomorphic to $\mathbb{F}_{p^l}$. It is known that $\mathbb{F}_{p^k}$ can be embedded into $\mathbb{F}_{p^l}$ if and only if $k$ divides $l$. So every $\mathbb{F}_{p^k}$ with $k$ dividing $l$ can be embedded into $DN_g(q,n).$ Note that they have different multiplication, but $\mathbb{F}_{p^k}$ is isomorphic to a subfield of $\mathbb{F}_{p^l}$ which is a subnearfield of $DN_g(q,n).$ Hence  $\mathbb{F}_{p^k}$ is a  subnearfield of $DN_g(q,n).$ Observe that $k$ must divide $l$, dividing $l \cdot n$ is not enough. \\

 Furthermore, if two of $\alpha, \beta , \alpha +\beta $ belong to $g^{[s]_q}H$  and the  third to $g^{[t]_q}H$ such that $\textsc{gcd}(t+n-s,n)=1$ where  $s \neq t$ then by proof of Theorem \ref{thmm} $D(\alpha,\beta)=\mathbb{F}_{p^{l\cdot \textsc{gcd}(t+n-s,n)}}$. Suppose $\textsc{gcd}(t+n-s,n) \neq 1$ then $l \cdot \textsc{gcd}(t+n-s,n)$ can not divide $l$. Hence $D(\alpha, \beta)$ can not be a subnearfield of $D(R)=\mathbb{F}_{p^l}$.
\end{rem}

\section{Further results on  $R$-subgroups } 

 Let $R$ be a finite nearfield. All the results in this section make use of  Theorem \ref{th2} where at each \say{\textsl{distributivity trick}}, a triple $ (\alpha ,\beta ,\lambda) \in R^3$ such that $\lambda \notin D(\alpha, \beta)$ is chosen within an application of the eGe algorithm.
\subsection{Some  properties}
In this subsection, we study the  notion of $R$-dimension and $R$-basis of  $R$-subgroups.

\begin{defn} A finite set $V=\{v_1,\ldots, v_k\}$ of non-zero vectors in $R^n$   is $R$-linearly dependent if there exists $v_i\in V$ such that $v_i\in gen(v_1,\ldots,\hat{v_i},\ldots,v_k)$.
\end{defn}
Note that we use $\{ v_1,\ldots,\hat{v_i},\ldots,v_k \}$ to denote the fact that the vector $v_i$ has been removed from the set of vectors $\{v_1, \ldots, v_k \}.$

\begin{lem}Let $R$ be a finite nearfield and $v_1,\ldots,v_k\in R^m$. Then
$  \big \vert gen(v_1,\ldots,v_k) \big \vert = \vert R \vert ^{k'} $ where $k'$ is the number of non-zero rows obtained after performing the eGe algorithm on the vectors $v_1,\ldots,v_k$. 
\end{lem}
\begin{proof}
By Theorem \ref{th2}, we have  $gen(v_1,\ldots,v_k)= \bigoplus_{i=1}^{k'}u_iR $ where $U=\big(u_{i}^j  \big) \in R^{k' \times m}$ is the final matrix after the expanded Gaussian  elimination with the property that all its columns have at most one non-zero entry. Hence 
 $\big \vert gen(v_1,\ldots,v_k)  \big \vert  =  \big \vert \bigoplus_{i=1}^{k'}u_iR  \big \vert = \vert u_1R \vert \times \vert u_2R \vert \times  \cdots \times  \vert u_{k'}R \vert = \vert R \vert^ {k'}.$
\end{proof}
In particular, we obtain:
\begin{cor}
Let $R$ be a finite nearfield and $T$  an $R$-subgroup of $R^m$. Then $ \vert T \vert = \vert R \vert ^k $ for some $k \leq m$.
\end{cor}

In analogy to the notion of a basis of a subspace in the theory of vector spaces, we introduce what we will call $R$-basis and $R$-dimension of an $R$-subgroup of the finite dimensional Beidleman near-vector spaces $R^m$. 
\begin{defn}Let $R$ be a finite nearfield
and  $T$  an $R$-subgroup of $R^m$. There exist some vectors $v_1,\ldots,v_k\in R^m$ such that $gen(v_1,\ldots,v_k)=T$. By the eGe algorithm, the finite set of non-zero row vectors  $ \lbrace u_1, \ldots, u_{k'} \rbrace$ obtained will be called an $R$-basis of $T$ and the number $k'$ will be  called the $R$-dimension of $T$.
\label{df}
\end{defn}
\begin{rem}Let $R$ be a finite nearfield.
Suppose there exist $v_1,\ldots,v_k \in R^m$ and $w_1,\ldots,w_l \in R^m$ such that $gen(v_1,\ldots,v_k )=gen(w_1,\ldots,w_l)=T$. By Theorem \ref{th2} we have $gen(v_1,\ldots,v_k)=\bigoplus_{i=1}^{k'}\mu_iR$ and $gen(w_1,\ldots,w_l)=\bigoplus_{i=1}^{l'}\nu_iR$ such that $\big ( \mu_i^j \big )_ {\substack{ 1 \leq i \leq k' \\ 1 \leq j \leq m}}$  and $\big ( \nu_i^j \big )_ {\substack{ 1 \leq i \leq l' \\ 1 \leq j \leq m}}$ are some matrices that have at most one non-zero entry in each column. We have, $\vert gen(v_1,\ldots,v_k) \vert = \vert gen(w_1,\ldots,w_l) \vert.$ Then $ \vert R \vert^{k'}=\vert R \vert^{l'}.$ Thus $k'=l'$. Thus  the $R$-dimension of $T$ is well-defined.
\end{rem}
\begin{defn}Let $R$ be a finite nearfield
and  $T$  an $R$-subgroup of $R^m$. $V$ generates $T$ if $V\subseteq T$ and $gen(V)=T$.
\end{defn}
\begin{defn}Let $R$ be a finite nearfield
and  $T$  a $R$-subgroup of $R^m$.
 $V$ is a seed set for $T$ if $V$ is $R$-linearly independent and $V$ generates $T$.
\end{defn}
\begin{defn}Let $R$ be a finite nearfield
and  $T$  an $R$-subgroup of $R^m$.
The seed number of $T$ is the minimal cardinality  of all the seed sets, i.e., $s(T)= \min _{V\in \mathcal{B}} |V|$ where $\mathcal{B}$ is the set of seed sets for $T$. 
\end{defn}

\begin{rem}Let $R$ be a finite nearfield
and  $T$  an $R$-subgroup of $R^m$. Then 
 the seed number $s(T)$ is well-defined. Note that for every $R$-subgroup $T$, $gen(T)=T$ (i.e., $T$ is generated by all its elements), then $s(T) \leq |T|$. Thus $s(T)$ is well-defined. 
\end{rem}
\begin{lem}Let $R$ be a finite nearfield
and  $T$  an $R$-subgroup of $R^m$. Then
\begin{align*}
R\mbox{-}\dim (T)= \max _{V\in \mathcal{B}} |V|
\end{align*}
 where $\mathcal{B}$ is the set of seed sets for $T$. 
\end{lem}
\begin{proof} Let  $T$  be an $R$-subgroup of $R^m$.
Suppose $V$ is a seed set for $T$ containing $k$ vectors. Then $gen(V)=T$. By the eGe algorithm we have $gen(V)=\bigoplus_{i=1}^{k'}u_iR $ where $\big ( u_i^j \big )_{\substack{ 1 \leq i \leq k' \\ 1 \leq j \leq m}}$  is a  matrix that has at most one non-zero entry in each column. Since $V$ is $R$-linearly independent,  $k' \geq k.$ Thus $R$-$\dim(T) \geq \vert V \vert.$ 
\end{proof}
\begin{exa}~Let us consider the finite dimensional Beidleman near-vector space $(R^m,R)$ where $R$ is a finite nearfield. Suppose $m=1.$ Then $s(R)=1.$
Suppose $n=2.$ Since $gen(v)=vR$ for all $v \in R^2,$ we have $s(R^2) \neq 1$. Thus  $s(R^2)=2$. Suppose $m=3$. Let $v_1=(1,1,0)$ and $v_2=(1,0,1)$  in $R^3$ such that $gen(v_1,v_2)=gen(e_1,e_2,e_3)=R^3$
 where $e_1=(1,0,0), e_2=(0,1,0)$ and $e_3=(0,0,1)$. So $\{ e_1,e_2, e_3 \}$ and $\{ v_1,v_2 \}$ are some seed sets for $R^3$. A seed set of maximum size is $\{e_1,e_2,e_3 \}$ and a seed set of minimum size is $\{v_1, v_2 \}$. Thus  $s(R^3)= 2  $ and $ R$-$\dim(R^3)=3.$
\end{exa}

\begin{lem}
Let $R$ be a finite nearfield and $V$  a finite set of vectors in $R^m.$ If $gen(V)=T$, then $V$ contains a seed set for $T$.
\end{lem}

\begin{proof} 
  Let $V= \{ v_1, \ldots , v_k\}$ such that $gen(V)=T$. Using Lemma \ref{l1},  keep removing elements from $V$ that do not contribute to  $gen(V)$ until we find a seed set for $T$. Note that the order in which we remove elements matters, in the sense that the seed sets we get at the end may have different sizes.
\end{proof}

\begin{lem}
Let $R$ be a finite nearfield and $V$  a finite set of vectors in $R^m$ such that $gen(V)=T$. Assume that the vectors in $V$ are arranged in a matrix $M$ of size $k \times m$. Then $s(T)$ is less than or equal to the  number of pivots that we  get in the reduced row echelon form of $M$.
\end{lem}
\begin{proof}
Let $V$ be a seed set for $T$. Suppose $V$ contains the vectors $v_1,\ldots, v_k \in R^m$ arranged in a  matrix $M \in R^{k \times m}.$ The reduced row echelon form of $M$ gives another matrix $M' \in R^{k \times m}$ whose  set   of  non-zero row vectors is $W= \{w_1, \ldots, w_t \}$ where $t \leq k.$ By Lemma \ref{l1}, we have $T=gen(V)=gen(W)$ and  $W$ is $R$-linearly independent. So  $W$ is also a seed set for $T$. If $W$ is minimum sized then $s(T)$ is the  number of elements in $W$ which correspond to the number of pivots of $M'=RREF(M).$ Else $s(T)$  is less than the  number of elements in $W$.
\end{proof}

\subsection{Seed number of $R^m$}
%

Let $R$ be a nearfield and  $T$ an $R$-subgroup of $R^m$ where $m$ is a positive integer. By definition $s(T)$ is the minimum size of all seed sets of $T.$ In the theory of vector spaces, the dimension of a subspace is at most the dimension of the entire space. Similarly, the $R$-dimension of an $R$-subgroup of $R^m$ is at most the $R$-dimension of the entire space. Thus $s(T) \leq R$-$\dim(T) \leq m.$

In analogy to the theory of  vector spaces we have the following.

\begin{lem}
Let $v=(v_i)_{1 \leq i \leq m},w=(w_i)_{1 \leq i \leq m} \in R^m$. Suppose there exists $(v_j,w_j)=\rho(v_i,w_i)$ where $i \neq j$ and $\rho \in R.$ Then by elimination of one of the pairs $(v_i,w_i)$ or  $(v_j,w_j)$ from $v$ and $w$, we obtain the new vectors 
$v'= (v_1,\ldots,\hat{v_i},\ldots,v_{m}) \in R^{m-1}, w'= (w_1,\ldots,\hat{w_i},\ldots,w_{m}) \in R^{m-1}$ or $ v'= (v_1,\ldots,\hat{v_j},\ldots,v_{m}) \in R^{m-1},  w'= (w_1,\ldots,\hat{w_j},\ldots,w_{m}) \in R^{m-1} $ with  
 \begin{align*}
 R \mbox{-} \dim  \big ( gen(v,w) \big ) = R \mbox{-} \dim  \big  ( gen(v',w') \big ).
 \end{align*}
\label{lem}
\end{lem}
\begin{proof}
Let $v=(v_i)_{1 \leq i \leq m},w=(w_i)_{1 \leq i \leq m} \in R^m$ arranged in a matrix 
\begin{align*}
V=\begin{bmatrix}
   v_1&v_2& v_3    & \ldots & v_m  \\
   w_1&w_2 & w_3   & \ldots & w_m \\
\end{bmatrix} \in R^{2 \times m}.
\end{align*}
Suppose for simplicity that $(v_2,w_2)= \rho (v_1,w_1)= (\rho v_1, \rho w_1) $ where $v_1,v_2,w_1,w_2, \rho \in R^*.$ We apply the eGe algorithm on the set $\{v,w \}$. Note that  eliminating the non-zero entries in the first column of $V$ will also automatically eliminate  the non-zero entries in the second column of $V$. Hence, the number of additional rows created  as we apply eGe algorithm on $\{v,w \}$ is the same  as the number of additional rows we will create when we apply the eGe on $\{v',w' \}$ where $v'=(v_2,\ldots,v_m) \in R^{m-1}, w'=(w_2,\ldots,w_m) \in R^{m-1}.$ 
\end{proof}

We now deduce the following:
\begin{thm} Let $R$ be a finite Dickson nearfield that arises from the Dickson pair $(q,n)$ and $T$ be an $R$-subgroup of $R^m$. If $s(T)=2$, then 
\begin{align*}
2\leq R\mbox{-}\dim (T) \leq  q^n +1.
\end{align*}	
\label{th}
\end{thm}
\begin{proof}
Suppose  $T$ is an $R$-subgroup of $ R^m$ and $s(T)=2$. Then there exist  minimal generators of $T$ denoted as $v_1=(v_1^1,\ldots,v_1^m), v_2=(v_2^1,\ldots,v_2^m)$, i.e., $gen(v_1,v_2)=T$ where $v_1$ and $v_2$ are arranged in a matrix $V \in R^{2 \times m}.$ 

Consider the pairs $(v_1^i,v_2^i)$ with $v_1^i \neq0 \neq v_2^i$ where $1 \leq i \leq m$. Assume there are two such pairs such that
\begin{align}
v_1^j= \rho v_1^i, \thickspace v_2^j= \rho v_2^i \thickspace \mbox{for} \thickspace  i \neq j \thickspace \mbox{ and} \thickspace  \rho \in R.
 \label{no1}
\end{align}
 Note that we may think of $(v_1^j,v_2^j)$ as a multiple of $(v_1^i,v_2^i)$. For simplicity, let  $i=1,j=2$.
Thus we can write $v_1^2= \rho v_1^1$ and $v_2^2=\rho v_2^1$.

Now the expanded Gaussian elimination algorithm (in the proof of Theorem \ref{th2}) returns a seed set $ \{ u_1,\ldots,u_k \}$ (where $k=R$-$\dim(T)$) for $T$ such that every column of the matrix with rows $u_1,\ldots,u_k$ has at most one non-zero entry.
In this process, we take $R$-linear combinations of the rows to form new rows, e.g., suppose we have $\alpha,\beta,\lambda \in R$ and we consider the vector $z_1=(v_1\alpha + v_2 \beta)\lambda$. So $z_1 \in LC_2(v_1,v_2)$.
Then $z_1^1=(v_1^1 \alpha + v_2^1 \beta)\lambda$ and	
\[
z_1^2=(v_1^2 \alpha + v_2^2 \beta)\lambda=(\rho v_1^1\alpha +  \rho v_2^1 \beta)\lambda\\
=\rho (v_1^1 \alpha + v_2^1 \beta)\lambda= \rho z_1^1.
\] 

Thus every vector created in this manner satisfies $z_1^2= \rho z_1^1$. In fact for $z_1 \in LC_m(v_1,v_2)$ where $m$ is a positive integer, we will still have that $z_1^2=\rho z_1^1$.  At the end of the expanded Gaussian elimination,  $u_1$ is an  $R$-linear combination of $v_1$ and $v_2$, so we will have $u_1^1=1, u_1^2=\rho $ where the column is indicated by the super-script.

 Furthermore, if $v_1^i=0$ or $v_2^i=0$ then the pair $(v_1^i,v_2^i)=(0,v_2^i)=v_2^i(0,1)$ or $(v_1^i,v_2^i)=(v_1^i,0)=v_1^i(1,0).$ Else $v_1^i \neq 0 \neq v_2^i$ then $(v_1^i,v_2^i)=v_1^i \big (1, (v_1^i)^{-1}v_2^i \big )=v_1^i(1,r)$ where $r=(v_1^i)^{-1}v_2^i.$ So we can write any pair as a multiple of one of the following pairs : $(1,0),(0,1)$ or $(1,r)$ for $r \in R^*$ (which can't be  expressed as  multiples of other pairs in $R^2$). For example, suppose $(\rho_1,0)=\rho_1(1,0), (0,\rho_3)=\rho_3(0,1)$ and $(\rho_2,\rho_4)=\rho_2(1,r)$ where $r=\rho_2^{-1}\rho_4,\rho_1,\rho_2,\rho_3,\rho_4 \in R^*$. Then
\begin{align*}
 R \mbox{-}\dim  \big (  gen((\rho_1,0,\rho_2), (0,\rho_3,\rho_4)) \big ) = R \mbox{-} \dim \big (  gen((1,0,1), (0,1,r))\big ).
\end{align*}
  By Lemma \ref{lem}, for any two pairs $(v_1^i,v_2^i)$ and $(v_1^j,v_2^j)$  for $i \neq j$ satisfying the above condition  (\ref{no1}), we may eliminate one of them without changing the $R$-dimension of the $R$-subgroup generated by $v_1,v_2$. Hence the  $R$-dimension of $T$  cannot exceed the maximal number  of pairs  where we can not eliminate any other pairs of the form  $(1,0),(0,1)$ or $(1,r)$ for $r \in R^*$. Thus 
\begin{align*}
\max_{v_1,v_2 \in R^m}  \left \{ R \mbox{-}\dim \big (gen(v_1,v_2) \big )  \right \}=2+(|R|-1)=1 + |R|=q^n + 1.
\end{align*}

%

\end{proof}

\begin{rem}
Let $T$ be an $R$-subgroup of $R^m$ for a positive integer $m.$ Assume that $s(T)=2$. By Theorem \ref{th} the $R$-dimension of $T$ is restricted to a specific range such that $s(T) \leq R$-$\dim(T) \leq |R|+1.$  
\end{rem}
\begin{exa} 
\label{ex2} As in (\cite{pilz2011near}, p.257), let us consider the finite field $(\mathbb{F}_{3^{2}}, +, \cdot) $ with
\[\mathbb{F}_{3^{2}} := \{0,1,2,x,1+x,2+x,2x,1+2x,2+2x\},\]
where $x$ is a zero of $X^{2}+1 \in \Bbb{Z}_{3}[X]$ with the new multiplication defined as

$$
a \circ b := \left\{\begin{array}{cc}
              a \cdot b     & \mbox{ if $a$  is a square in ($\mathbb{F}_{3^{2}}$, $+$, $\cdot$)}\\
              a \cdot b^3 & \mbox{ otherwise }
              \end{array}
\right.
$$
This gives the smallest finite  Dickson nearfield $R=DN_g(3,2)$ which is not a field  (here $\mathbb{F}_{3^{2}}^* =\big < g \big > $ and $H =\big < g^2 \big >$).
Let $v=(x + 2, x + 1, 1, 1, 2x + 1, x + 1, 2, 0, 2x + 1, 2x + 2, x + 2,
        2x, 2x + 1, 2x + 2, x + 1, x + 2, 2x + 1, 2x + 2, x + 1)$ and $w=(2x, 2, 1, 2x + 2, 0, 2x, 2, 2x, 2x + 1, x + 1, 0, 2x + 1, 1,
        x + 1, 2x + 2, 2x + 2, 2x + 2, 2x + 2, x + 1) \in R^{19}$. Then  $T=gen(v,w)$ is an $R$-subgroup of $R^{19}$ and
we have $R$-$ \dim (T)=10$.
\end{exa}
\begin{rem}
In contrast to the theory of vector spaces, the following does not always hold:
$
R\mbox{-}\dim \big ( gen(v_1, \ldots, v_k) \big ) = R\mbox{-}\dim \big ( gen(w_1, \ldots, w_m) \big )$ where  a matrix $V \in R ^{k \times m}$ contains the rows $v_1, \ldots, v_k$  and the columns $w_1, \ldots, w_m$. To see this, suppose $R \in DN(3,2)$ (from Example \ref{ex2}). Let $v_1=(1,2,x,0,0), v_2=(0,0,0,1,0),v_3=(1,0,0,0,1) $ and $w_1=\begin{bmatrix}
1\\
0\\
1
\end{bmatrix}, w_2=\begin{bmatrix}
2 \\
0\\
0
\end{bmatrix}, w_3=\begin{bmatrix}
x \\
0\\
0
\end{bmatrix},w_4=\begin{bmatrix}
0\\
1\\
0
\end{bmatrix}, w_5=\begin{bmatrix}
0\\
0\\
1
\end{bmatrix}.$ In fact $R$-$\dim \big ( gen(v_1,v_2,v_3)\big )=4$ but $R$-$\dim \big ( gen(w_1,w_2,w_3,w_4,w_5)\big )=3.$
\end{rem}

In the next  theorem we shall determine the seed number of the finite dimensional Beidleman near-vector space $R^m$ where $R$ is a finite Dickson nearfield. This is accomplished by finding   two  vectors $v,w \in R^m$ such that $gen(v,w)=R^m$ for some positive integer $m$ such that $m \leq |R|+1.$ The result is the converse of Theorem \ref{th}.
\begin{thm} Let $R$ be a finite Dickson nearfield that arises from the Dickson pair $(q,n)$.
For every value $m$ satisfying $ 2\leq m\leq q^n+1,$ we have $s(R^m)=2.$ 
\end{thm}
\begin{proof}For $2 \leq m \leq |R|+1$, we choose $v=(1,0,1, \ldots,1)$ and $w= (0,1,w^3,\ldots, w^m) \in R^m$ arranged in a matrix 
\begin{align*}
V=\begin{bmatrix}
   1&0& 1    & \ldots & 1  \\
   0&1 & w^3   & \ldots & w^m \\
\end{bmatrix} \in R^{2 \times m},
\end{align*}
 where each element $w^j \neq 1$ for $j \in \{3, \ldots ,m \}$ is a non-zero distinct element (i.e, $w^j \in R^*\smallsetminus \{1\}$). Note that all the  pairs   
 \begin{align*}
 (1,0), (0,1), (1,w^3), \ldots, (1,w^j), \ldots, (1,w^m)
 \end{align*}
satisfy  the following condition
\begin{align}
 \thickspace (v_j,w_j) \neq \alpha (v_i,w_i) \thickspace \mbox{for all} \thickspace i,j \thickspace \mbox{where} \thickspace i \neq j \thickspace \mbox{and} \thickspace \alpha \in R^*.
 \label{no}
\end{align}

 Thus by implementing the  eGe algorithm (the explicit  procedure in the proof of Theorem \ref{th2}) on the initial matrix $V$, we create at least $m-2$ additional rows (since the first two columns each have exactly one non-zero entry). Hence we may apply the \say{\textsl{distributivity trick}} on the third column  and so on.  Note that the process will stop after creating exactly $m-2$ new rows so that at the end of eGe, we will get in total $m$ rows so that $gen(v,w)= \bigoplus_{i=1}^me_i R =R^m$ where $\{e_i\}_{i=1,\ldots,m }$ is the standard basis. Thus $s(R^m)=2.$
\end{proof}
\begin{exa}
Let $R \in DN(3,2)$. There exist $v=(1,2x+2,x,0,x)$ and $w= (2,2x,1,2,x) \in R^5$ such that $gen(v,w)=R^5.$ Thus  $s(R^5)=2$. 
\end{exa}

We have seen that it takes two vectors that belong to   $R^m$ under the condition that $m \leq |R| +1$ to generate the whole space.

%
%
%
%
%

\vspace{10mm}

\section{Concluding comments}

Let $R$ be a finite Dickson nearfield that arises from the Dickson pair $(q,n)$.
We have determined $D(\alpha, \beta)$ for a given $(\alpha,\beta) \in R^2$ and  showed that it is used in the determination of $s(R^m)$ by the implementation of the eGe algorithm. Some computational methods have  been  implemented in Sage. Note that the proof of Theorem \ref{th2} taken from \cite{djagbahowell18} was included to illustrate to the readers the eGe algorithm and the motivation of the generailised distributive set.  In contrast to the situation for $D(R)$ by the work of Zemmer in \cite{zemmer1964} (see Theorem \ref{tp}) we found that $D(\alpha, \beta)$ is not always a subnearfield of $R$. It shouldn't be too hard to use the characterization of the $7$ exceptional finite nearfields to determine their generalized distributive sets. Furthermore if $n>2$ and $\alpha, \beta, \alpha + \beta$  belong all to distinct $H$-cosets then  the multiplicative center of $D(\alpha, \beta)$ can't be characterized since $D(\alpha, \beta)$ is not always a multiplicative subgroup of $R$. It seems appropriate to close with further problems on the generalized distributive set and seed number.
\begin{itemize}
\item Let $(q,n)$ be a Dickson pair with $q=p^l$ for some prime $p$ and integers $l,n$  such that $n >2.$  Let $g$ be a generator of $\mathbb{F}_{q^n}^*$ and  $R$ the nearfield constructed with $H = \big < g^n \big >.$ Let  $ (\alpha, \beta , \alpha + \beta )\in g^{\frac{q^r-1}{q-1}}H \times g^{\frac{q^s-1}{q-1}}H \times g^{\frac{q^t-1}{q-1}}H$ such that $r,s, t$ are all distinct. Can we find a condition on  $(q,n)$ such that $D(\alpha,\beta)$ is  always a subfield of $\mathbb{F}_{q^n}$?

\item  Can we find a necessary and sufficient condition on the Dickson pair $(q,n)$ such that $D(\alpha,\beta)$ is always  a subnearfield of $DN_g(q,n)?$

\item  Let $T$ be an $R$-subgroup of $R^m.$ In the general setting if $s(T)=k$ what are the possible $R$-dimensions for $T$?
\end{itemize}

\section{Acknowledgments}
I thank Dr Gareth Boxall and  Dr Karin-Therese Howell for their advice  on this work. I thank as well Georg Anegg for his collaborations. I  worked on this paper while studying toward my PhD at Stellenbosch University.

\section{Funding}
 I am  grateful  for  funding  by  AIMS (South  Africa) and DAAD. This work is based on the research supported in part by the National Research Foundation of South Africa (Grant Numbers $93050,96234$). 
\section{Appendix}
Let $R $ be a finite Dickson nearfield that arises from the Dickson pair $(q,n)$. In this section we give more details about the pseudo-code that we have implemented in Sage for some tests on $D(\alpha,\beta)$ for a given $(\alpha, \beta) \in R^2$. $DSS$ (Distributive Set  Subfields) tests if $D(\alpha, \beta)$, for a given pair $(\alpha, \beta) \in g^{[r]_q}H \times g^{[s]_q}H $ for some positive integers $r,s$, can be  considered  as  a subfield of $\mathbb{F}_{q^n}$. The algorithm is described as follows:

\begin{algo}~\\

\textbf{Step 1: Dickson pair}

\noindent\rule[0.5ex]{\linewidth}{0.5pt}

 We define the \textbf{function isDicksonpair}$(q,n)$. Note that isDicksonpair$(q,n)$ returns True if $(q,n)$ is Dickson pair, false if it is not and list all Dickson pairs $(p^l,n)$ such that $(p,n,l) \in (1,x) \times (1,y) \times (1,z)$ where $x,y$ and $z$  are three positive integers.

\noindent\rule[0.5ex]{\linewidth}{0.5pt}
%
%
%

 Denote by $H$ the subgroup of $\mathbb{F}_{q^n}^*$ generated by $g^n$. Each non-zero element $\alpha $ of $\mathbb{F}_{q^n}^*$ is then in an $H$-coset of the form $g^{\frac{q^j-1}{q-1}}H$. The next function computes the value of $j$ in $\{0,1,...,n-1\}$ for $\alpha \in g^{\frac{q^j-1}{q-1}}H $.

\textbf{Step 2: Index of $H$-coset}

\noindent\rule[0.5ex]{\linewidth}{0.5pt}

\textbf{Input:} An element $\alpha $ in $\mathbb{F}_{q^n}^*.$ 

\textbf{Output:} The index $j$ of $\alpha$ if it exists, false if it doesn't.

\noindent\rule[0.5ex]{\linewidth}{0.5pt}

define \textbf{function isDicksonpair}$(q,n)$
\begin{enumerate}
\item $\mathbb{F}_{q^n}^*= \big <g \big >$
\item define \textbf{function associatedcoset}$(\alpha)$
\item \quad \textbf{if} $\alpha \neq 0$
\item \quad \quad $idx= \log(\alpha,g)$  (return the logarithm of $\alpha$ to the base of $g$).
\item \quad \textbf{else} \textbf{return} False
\item \quad$idx=idx \mod (n)$ (since $H=<g^n> $ we only care about $idx$ modulo $n$).
\item \quad \textbf{for} $j \in \{ 0,\ldots,n \}$ \textbf{do}
\item \quad \quad $k =\frac{q^j-1}{q-1}$
\item \quad \quad $k=k \mod (n)$
\item \quad \quad \textbf{if} $idx=k$ \textbf{then}
\item \quad  \quad \quad \textbf{return} $j$ 
\end{enumerate}

\textbf{Step 3: Matrix basis}

\noindent\rule[0.5ex]{\linewidth}{0.5pt}

\textbf{Input:} $\alpha,\beta \in \mathbb{F}_{q^n}.$

\textbf{Output:} Basis for the space of solutions $\lambda$ to the equation $\phi(\lambda)=0$. The idea is to perform this test on various randomly chosen $\alpha,\beta \in \mathbb{F}_{q^n}^*$.
 
\noindent\rule[0.5ex]{\linewidth}{0.5pt}

\begin{enumerate}
\item define \textbf{function }$\phi(\alpha,\beta,\lambda,r,s,t)$
\item \quad \textbf{return} $ \alpha(\lambda^{q^t}-\lambda^{q^r}) - \beta(\lambda^{q^s}-\lambda^{q^t})$
\item  define  \textbf{function performtest}$(\alpha,\beta)$
\item \quad $r= associatedcoset(\alpha)$
\item \quad $s= associatedcoset(\beta)$
\item \quad $t= associatedcoset(\alpha+\beta)$
\item \quad \textbf{if} $r=s$ or $r=t$ or $s=t$ \textbf{then}
\item \quad \quad \textbf{return} " $D(\alpha,\beta)$ is finite field"
\item \quad \textbf{else}
\item \quad \quad \textbf{for} $j \in \{ 0,\ldots,n-1 \}$ \textbf{do}
\item \quad \quad \quad $M \longleftarrow $ matrix associated to $\phi(\alpha,\beta,g^j,r,s,t)$
\item \quad \quad \quad $K \longleftarrow $ kernel $(M)$
\item \quad \quad \textbf{if} $\dim (K)>1$ \textbf{then}
\item \quad \quad \quad \textbf{return} $K$

\item \textbf{if} isDicksonpair$(q,n)$ \textbf{then}
\item \quad \textbf{for} $r_1 \in \{1,\ldots,q^n -1 \}$  \textbf{do}
\item \quad \quad \textbf{for} $r_2 \in \{1,\ldots,q^n -1 \}$  \textbf{do}
\item \quad \quad \quad $\alpha \leftarrow g^{r_1}$
\item \quad \quad \quad $\beta  \leftarrow g^{r_2}$
\item \quad \quad \quad performtest$(\alpha,\beta)$ 
\end{enumerate}

\textbf{Step 4: Row vector to field element}

\noindent\rule[0.5ex]{\linewidth}{0.5pt}

\textbf{Input:} A row vector from the  matrix basis $K$ of $D(\alpha,\beta)$ over $\mathbb{F}_q.$

\textbf{Output:} Field element of $D(\alpha,\beta)$.
 
\noindent\rule[0.5ex]{\linewidth}{0.5pt}

\begin{enumerate}
\item define \textbf{function rowvectortofieldelement}$(v)$
\item \quad $a \longleftarrow 0$
\item \quad $k \longleftarrow$ number of column in $v$  (length $(v)$).
\item \quad \textbf{for} $i \in \{ 1,\ldots k \}$
\item \quad  \quad $a \longleftarrow \sum_{i=1}^{k} v[i]* g^{i-1}$
\item \quad \textbf{return} a
\end{enumerate}

\textbf{Step 5: Field test}

\begin{enumerate}
\item define \textbf{function isfield}($K$)
\item \quad \textbf{if} $\dim (K)=1 $ \textbf{then}
\item \quad \quad \textbf{return} True
\item \quad \textbf{else}
\item \quad \quad $B \longleftarrow $ set  constituted by rows vectors in $K$.
\item \quad \quad \textbf{for} $i \in B$ \textbf{do}
\item \quad \quad \quad \textbf{for} $j \in B$ \textbf{do}
\item \quad \quad \quad \quad $b_1\leftarrow rowvectortofieldelement(i) $
\item \quad \quad \quad \quad $b_2 \leftarrow rowvectortofieldelement(j) $
\item \quad \quad \quad \quad $b \leftarrow b_1b_2^{-1}$
\item \quad \quad \quad \quad \textbf{if} row vector$(b) \notin K$ \textbf{then}
\item \quad \quad \quad \quad \quad \textbf{return} False
\end{enumerate}

\noindent\rule[0.5ex]{\linewidth}{0.5pt}

\textbf{Input:} $K=Kernel(M).$  

\textbf{Output:}  $D(\alpha,\beta)$ is a finite field or not.
 
\noindent\rule[0.5ex]{\linewidth}{0.5pt}

\begin{enumerate}
\item \textbf{if} isfield(K) \textbf{then}
\item \quad \textbf{write} "$D(\alpha, \beta)$ is a finite field".
\item \textbf{else}
\item  \quad \textbf{write} "$D(\alpha, \beta)$ is not a finite field".
\end{enumerate}
\label{algo}
\end{algo}
%
%
%

\end{document}